\documentclass[11pt,a4paper]{amsart}
\usepackage{amsfonts,amsthm, amscd, epsfig, amsmath, amssymb,enumerate,bbm}

\usepackage{hyperref}
\hypersetup{
    colorlinks=true,
    linkcolor=blue,
    citecolor=red,
    pdfpagemode=FullScreen,
    pdftitle={Aletti,Crimaldi,Ghiglietti}
    }

\usepackage{graphicx}
\usepackage{url}
\usepackage{bbm}

\usepackage{tikz}
\foreach \a in {q,w,e,r,t,y,u,i,o,p,a,s,d,f,g,h,j,k,l,z,x,c,v,b,n,m,Q,W,E,R,T,Y,U,I,O,P,A,S,D,F,G,H,J,K,L,Z,X,C,V,B,N,M}{
  \expandafter\xdef\csname v\a\endcsname {
		{\noexpand\boldsymbol{\a}}
	}
}

\newcommand{\vone}{{\boldsymbol{1}}}
\newcommand{\vzero}{{\boldsymbol{0}}}

\theoremstyle{plain}
\newtheorem{theorem}{Theorem}[section]
\newtheorem{lemma}[theorem]{Lemma}
\newtheorem{prop}[theorem]{Proposition}

\theoremstyle{remark}
\newtheorem{remark}{Remark}

\theoremstyle{assumption}
\newtheorem{assumption}{Assumption}

\theoremstyle{definition}
\newtheorem{definition}{Definition}[section]
\newtheorem{example}{Example}[section]

\allowdisplaybreaks[3]

\begin{document}

\title[]{Networks of reinforced stochastic processes:\\
 probability of asymptotic polarization\\
  and related general results}

\author[G. Aletti]{Giacomo Aletti}
\address{ADAMSS Center,
  Universit\`a degli Studi di Milano, Milan, Italy}
\email{giacomo.aletti@unimi.it}

\author[I. Crimaldi]{Irene Crimaldi}
\address{IMT School for Advanced Studies Lucca, Lucca, Italy}
\email{irene.crimaldi@imtlucca.it}

\author[A. Ghiglietti]{Andrea Ghiglietti}
\address{Universit\`a degli Studi di Milano-Bicocca, Milan, Italy}
\email{andrea.ghiglietti@unimib.it (Corresponding author)}

\begin{abstract}
In a network of reinforced stochastic processes, for certain values of the parameters, all the agents' inclinations synchronize and converge almost surely toward a certain random variable. The present work aims at clarifying when the agents can \emph{asymptotically polarize}, i.e. when the common limit inclination can take the extreme values, 0 or 1, with probability zero, strictly positive, or equal to one.
Moreover, we present a suitable technique to estimate this probability that, along with the theoretical results, has been framed in the  more general setting of a class of martingales taking values in $[0,1]$ and following a specific dynamics. 
\\

\noindent {\em Key-words:} interacting random systems,
network-based dynamics, reinforced stochastic processes, urn models, 
martingales, polarization, touching the barriers, opinion dynamics, simulations. 
\end{abstract}

\maketitle

\section{Introduction: setting and scope}
\label{intro}

One of the main problems in Network Theory (e.g. \cite{alb-bar, new, hof}) is
to understand if the dynamics of the agents of the network will lead to some form of  
{\em synchronization} of their behavior (e.g. \cite{are}). 
A specific form of synchronization is the {\em polarization}, 
that can be roughly defined as the 
 positioning of all the network agents on one of two extreme opposing statuses. 
 The present work is placed in the recent stream of mathematical literature which studies 
the phenomena of synchronization and polarization for networks of agents 
whose behavior is driven by a {\em reinforcement} mechanism
(e.g. 
 \cite{ale-cri-ghi, ale-cri-ghi-MEAN, ale-cri-ghi-WEIGHT-MEAN, ale-cri-ghi-complete, 
 ale-ghi, ben, che-luc, cri-dai-lou-min, cri-lou-min, super-urn-2,  
 KauSah21, lima, sah}). 
 Specifically, we suppose to have a finite directed graph $G=(V,\, E)$, where 
 $V=\{1,...,N\}$, with $N\geq 2$, is the set of vertices, that is the network agents, 
 and $E\subseteq V\times V$ is the set of edges, where  
 each edge $(l_1,l_2)\in E$ represents the fact
 that agent $l_1$ has a direct influence on the agent $l_2$. We
 also associate a deterministic weight $w_{l_1,l_2}\geq 0$ to each pair
 $(l_1,l_2)\in V\times V$ in order to quantify how much $l_1$ can
 influence $l_2$ (a weight equal to zero means that the edge is not
 present). We define the matrix $W$, called in the sequel {\em
   interaction matrix}, as $W=[w_{l_1,l_2}]_{l_1,l_2\in V\times V}$
 and we assume the weights to be normalized so that $\sum_{l_1=1}^N
 w_{l_1,l_2}=1$ for each $l_2\in V$. Regarding the behavior of the agents, 
 we suppose that at each time-step they have to make a
choice between two possible actions $\{0,1\}$. For any $n\geq 1$, the
random variables $\{X_{n,l}:\,l\in V\}$ take values in $\{0,1\}$ and
they describe the actions adopted by the agents at time-step $n$. The dynamics 
is the following: for each $n\geq 0$, the random
 variables $\{X_{n+1,l}:\,l\in V\}$ are
 conditionally independent given ${\mathcal F}_{n}$ with
\begin{equation}\label{interacting-1-intro}
P(X_{n+1,l}=1\, \vert \, {\mathcal F}_n)=\sum_{l_1=1}^N w_{l_1,l} Z_{n,l_1}
\qquad\mbox{a.s.}\,,
\end{equation}
where, for each $l\in V$,
\begin{equation}\label{interacting-2-intro}
Z_{n,l}=(1-r_{n-1})Z_{n-1,l}+r_{n-1}X_{n,l},
\end{equation}
with $0\leq r_n<1$, $\{Z_{0,l}:\,l\in V\}$ random variables with
values in $[0,1]$ and ${\mathcal F}_n=\sigma(Z_{0,l}: l\in V)\vee
\sigma(X_{k,l}: 1\leq k\leq n,\,l\in V )$. Each random variable $Z_{n,l}$ takes values in
$[0,1]$ and it can be interpreted as the ``personal inclination'' of
the agent $l$ of adopting ``action 1''. Thus, Equation \eqref{interacting-1-intro} 
means that the probability that the
agent $l$ adopts ``action 1'' at time-step $(n+1)$ is given by a
convex combination of $l$'s own inclination and the inclination of the
other agents at time-step $n$, according to the ``influence-weights''
$w_{l_1,l}$. Note that we have a reinforcement mechanism for the
personal inclinations of the agents: indeed, by
\eqref{interacting-2-intro}, whenever $X_{n,l}=1$, we have
a strictly positive increment in the personal
  inclination of the agent $l$, that is $Z_{n,l}> Z_{n-1,l}$ (provided
  $Z_{n-1,l}<1$) and, in the case $w_{l,l}>0$ (which is the
  most usual in applications), this fact results in a greater
  probability of having $X_{n+1,l}=1$. 
  \\ \indent To express the
above dynamics in a compact form, let us define the vectors
$\boldsymbol{X}_{n}=(X_{n,1},..,X_{n,N})^{\top}$ and
$\boldsymbol{Z}_{n}=(Z_{n,1},..,Z_{n,N})^{\top}$.  Hence, for $n\geq
0$, the dynamics described by \eqref{interacting-1-intro} and
\eqref{interacting-2-intro} can be expressed as follows:
\begin{equation}\label{eq:cond-mean}
E[\boldsymbol{X}_{n+1}\vert \mathcal{G}_{n}]=W^{\top}\,\boldsymbol{Z}_{n},
\end{equation}
and 
\begin{equation}\label{eq:dynamics}
\boldsymbol{Z}_{n+1}\ =
\left(1-r_{n}\right)\boldsymbol{Z}_{n}\ +\ r_{n}\boldsymbol{X}_{n+1}.
\end{equation}
Moreover, the assumption about the normalization of the matrix $W$ can
be written as $W^{\top}\boldsymbol{1}=\boldsymbol{1}$, where
$\boldsymbol{1}$ denotes the vector with all the entries equal to $1$.
\\ \indent The recent paper \cite{ale-cri-ghi-complete} provides 
the sufficient and necessary conditions in order to have the
almost sure asymptotic synchronization of all the agents' inclinations, that is the
almost sure convergence toward zero of all the differences
$(Z_{n,l_1}-Z_{n,l_2})_n$, with $l_1,\,l_2\in V$. This phenomenon has been called 
{\em  complete almost sure asymptotic synchronization} and, in the considered setting, 
it is equivalent to the almost sure convergence of all the inclinations 
$(Z_{n,l})_n$, with $l\in V$, toward a certain common random variable $Z_\infty$. 
Under the assumption that $W$ is irreducible (i.e.~$G=(V,E)$ is a strongly connected graph)
and $P(\vZ_0\in\{\vzero,\vone\})<1$\footnote{Similarly to the
  notation $\vone$ already mentioned above, the symbol $\vzero$
  denotes the vector with all the entries equal to $0$.} (in order to
exclude the trivial initial conditions), in \cite{ale-cri-ghi-complete} it has been proven that:
\begin{itemize}
\item[(i)] when $W^{\top}$ is aperiodic, the complete almost sure asymptotic
  synchronization holds true if and only if $\sum_n r_n =+\infty$;
\item[(ii)] when $W^{\top}$ is periodic, the complete almost sure asymptotic
  synchronization holds true if and only if $\sum_n r_n(1-r_n) =+\infty$.
\end{itemize}
In the case of complete almost sure asymptotic synchronization, 
in order to provide a full description of the asymptotic dynamics of the network, 
we here deal with the phenomenon of 
non-trivial {\em asymptotic polarization}, i.e. 
with the question when the common random limit $Z_\infty$ can touch
the barrier-set $\{0, 1\}$ with a strictly positive probability, starting from 
$\vZ_0 \notin \{\vzero, \vone\}$. As we will show,  
this probability depends on how large is the weight of the new
information with respect to the present status of the process $(\vZ_n)_n$.
Indeed, looking at the dynamics
\eqref{eq:dynamics}, we can interpret the terms $r_n$ and $(1-r_n)$ as
the weights associated, respectively, to the new information
$\vX_{n+1}$ and to the present status $\vZ_n$, in the definition of
the next status $\vZ_{n+1}$ of the process. Moreover, the quantity
$\prod_{k=0}^n(1-r_k)$ can be seen as the weight associated to the
entire history of the process until time-step $n$, and so it can be
taken as a measure of the memory of the process at time-step
$n$. Under the conditions that ensure the complete almost sure
asymptotic synchronization (note that, in particular, this means that $\sum_n r_n=+\infty$), 
in the non-trivial case $P(\vZ_0 \notin
\{\vzero, \vone\})>0$, we can have different scenarios for the
probability of asymptotic polarization of the network. 
In particular,
adding the condition  
$r_n = O( \exp( - \sum_{k=0}^n r_k)\sum_{k=0}^n r_k  )$, or equivalently $r_n = O(  \prod_{k=0}^n(1-r_k) \sum_{k=0}^n r_k )$,
in order to bound the impact of the new information with respect to the past, 
we guarantee that the probability of non-trivial asymptotic polarization is zero (see Theorem \ref{th1}). 
On the contrary, if we add conditions in order to bound 
the impact of the history, we allow to have a strictly positive 
probability of non-trivial asymptotic polarization: specifically, we refer to
condition $\sum_n \prod_{k=0}^n (1-r_k)<+\infty$, which assures that the probabilities 
$P(Z_\infty=z|\vZ_0\notin \{\vzero,\vone\})$, with $z=0$ or $z=1$, are both 
strictly positive (see Theorem \ref{th2}).   
Finally, condition $\sum_n r_n^2<+\infty$ is enough to
avoid that the probability of asymptotic polarization is equal to
one (see Theorem \ref{th3}$(i)$). Indeed, this condition ensures that the weight of the new
information decreases to zero rapidly enough. Then, we can argue that,
since the contribution of the past in defining the next status of the
system remains relevant, the process $(\vZ_n)_n$ can stay close to its
initial value $\vZ_0$, and this, since $P(\vZ_0 \notin \{\vzero,
\vone\})>0$, forces $Z_\infty\in (0,1)$ with a strictly positive
probability.  On the contrary, when $\sum_n r_n^2=+\infty$, the limit
$Z_\infty$ touches the barriers with probability one and so it is a
Bernoulli random variable with parameter depending on the initial
random variable $\vZ_0$ (see Theorem \ref{th3}$(ii)$).  
In particular, the above results fully characterize the probability of 
non-trivial asymptotic polarization  in the case when 
 there exist $c>0$ and $0<\gamma\leq 1$ such that $\lim_n n^\gamma r_n=c$ and 
$\sum_n (r_n-cn^{-\gamma})$ is convergent, which is the setting of the results proven 
in \cite{ale-cri-ghi, ale-cri-ghi-MEAN, ale-cri-ghi-WEIGHT-MEAN, cri-dai-lou-min}.  
Table \ref{table-summary} summarizes the different scenarios according to 
the values for $\gamma$ and $c$.
\\

\begin{table}[ht]
\centering 
\begin{tabular}{c| c| c| c} 
\hline 
 Parameters & $0<\gamma\leq 1/2$ & $1/2<\gamma<1$ & $\gamma=1$ \\  
\hline 
$0<c\leq1$ & $=1$ & $\in(0,1)$ & $=0$ \\
\hline
$c>1$   & $=1$ & $\in(0,1)$ & $\in(0,1)$ \\
\hline 
\end{tabular}
\\[5pt]
\caption{Probability of non-trivial asymptotic polarization: 
possible scenarios for the case  when $\lim_n n^\gamma r_n=c$ and 
$\sum_n(r_n-cn^{-\gamma})$ is convergent. 
Specifically, when it is strictly positive, we have 
$P(Z_\infty=z| \vZ_0\notin\{\vzero,\,\vone\})>0$ for both $z=0$ and $z=1$.}
\label{table-summary}
\end{table}

\indent When the probability of non-trivial asymptotic polarization is in $(0,1)$,  
an interesting problem is to find statistical tools,  
based on the observation of the system until a certain time-step, in order to 
determine, up to a small probability, if the system will polarize in the limit. 
This paper deals with this question and provides a suitable technique, which is essentially 
based on concentration inequalities and Monte Carlo methods. Moreover, we 
use the provided estimators for the probability of asymptotic polarization to
define an asymptotic confidence interval for the random variable $Z_\infty$. 
The statistical tools illustrated in this work complete the more classical ones obtained in 
\cite{ale-cri-ghi, ale-cri-ghi-MEAN, ale-cri-ghi-WEIGHT-MEAN} 
by means of central limit theorems under the conditional probability
$P(\cdot \,|\, 0<Z_\infty<1)$. Indeed, when $Z_\infty$ takes values in $\{0,1\}$, these  
central limit theorems become convergences in probability to zero and so they are not useful in order to obtain 
the desired confidence interval for $Z_\infty$ under $P$.     
 The problem of making inference without excluding the case when the random limit $Z_\infty$
 belongs to $\{0,1\}$ is not covered by the urn model literature either. 
\\
 \indent Finally, we point out that we present   
 the theoretical results and the estimation technique  in the general setting of a stochastic process   
$M=(M_n)_{n\geq 0}$ that takes values in $[0,1]$ and is a martingale with respect to  
some  filtration
$\mathcal G=({\mathcal G}_n)_n$
 with the dynamics 
\begin{equation}\label{mart-dyn-intro}
M_{n+1}=(1-r_n)M_n+r_nY_{n+1},\quad n\geq 0,
\end{equation}
where $Y_{n+1}$ takes values in $[0,1]$ and $E[Y_{n+1}\vert {\mathcal
    G}_n]=M_n$ a.s. 
In particular, when
the probability of touching the barrier-set $\{0,1\}$ in the limit, given $0<M_0<1$, belongs to $(0,1)$,  
our scope is to find an estimator of this probability and 
  to construct an asymptotic confidence interval for the almost sure limit $M_\infty$ of $(M_n)_n$,  
  based on the information ${\mathcal G}_n$ collected until a certain time-step $n$. 
  This general framework can cover many other contexts in addition to the one presented above (e.g. \cite{LamPagTar04}).
\\
  
\indent The sequel of the paper is so structured. 
Section \ref{sec-mart} presents the results about the probability of touching the barrier-set $\{0, 1\}$ in the limit 
for a martingale with dynamics \eqref{mart-dyn-intro}.  In 
Section \ref{sec:barriers} we state the results of the asymptotic
polarization of a network of reinforced stochastic processes, i.e. 
we deal with the problem of touching
the barriers for the random variable $Z_\infty$ when the
complete almost sure asymptotic synchronization holds true.  
In Section~\ref{sec:estimation} we present the estimation technique 
for the probability of touching the barriers in the limit 
for a martingale with dynamics \eqref{mart-dyn-intro}.
Then, in Section~\ref{sec:interval} we construct a confidence interval for the limit random variable $M_\infty$,  
given the information collected until a certain time-step. 
Finally, in Section~\ref{sec:RSPs_simulations} the provided methodology is applied in the   
framework of a network of reinforced stochastic processes and
some simulation results are shown. In the appendix we give some recalls and technical details.


\section{Probability of touching the barriers in the limit for a class of martingales}
\label{sec-mart}

Consider a stochastic process $\mathcal{W} = (\mathcal{W}_n)_{n\geq 0}$ 
taking values in the interval $[0, 1]$ and following the
dynamics
\begin{equation}\label{rw-dyn}
\mathcal{W}_{n+1} = (1 - r_n )\mathcal{W}_n + r_n Y_{n+1},\quad n \geq 0,
\end{equation}
where $0\leq r_n<1$ and $Y_{n+1}$ takes values in $[0, 1]$. 
\\

The next proposition establishes a relationship between the above dynamics and 
the evolution of an urn model. 
In the particular case of $Y_{n+1}\in\{0,1\}$ and 
$E[Y_{n+1}|\mathcal{W}_0,Y_1,\dots,Y_n]=\mathcal{W}_n$,  
from this result we get that a 
single reinforced stochastic process corresponds to a time-dependent 
P\'olya urn \cite{pemantle-time-dependent}.  

\begin{prop}[Correspondence with an urn model]\label{th-corresp}
For each $n\geq 0$, the random variable $\mathcal{W}_n$ corresponds to the
proportion of balls of color $A$ inside the urn at time-step $n$
for a two-color urn process where the number of balls of color $A$
(resp. $B$) added to the urn at time-step $n\geq 1$ is $\alpha_{n}Y_n$
(resp. $\alpha_n(1-Y_n)$) with
\begin{equation}\label{eq:alpha_urn_model}
\alpha_{n} =
s_0\frac{r_{n-1}}{\prod_{k=0}^{n-1} (1-r_k)},
\end{equation}
where $s_0>0$ is an arbitrary constant.\\
\end{prop}

\begin{proof} Firstly, we recall the dynamics of a general two-color urn model: if
$S_0$ is the initial number of balls in the urn, 
$Z_n$ is the proportion of balls of color $A$ inside the urn at time-step $n$, 
$U_n^A$ (resp. $U_n^B$) is the number of balls of color $A$ (resp. $B$) 
added to the urn at time-step $n\geq 1$, we have that $(Z_n)_n$ follows the dynamics
\begin{equation}\label{urn-prop-dyn}
Z_{n+1}=(1-R_{n+1})Z_n+R_{n+1}Y_{n+1}
\end{equation}
with $R_{n+1}=U_{n+1}/S_{n+1}$ and $Y_{n+1}=U_{n+1}^A/U_{n+1}$, where
$U_{n+1}=U_{n+1}^A+U_{n+1}^B$ (i.e.~the number of balls added to the
urn at time-step $n+1$) and 
$S_{n+1}=S_0+\sum_{k=1}^{n+1} U_{k}$ (i.e.~ the total number of balls in the urn at time-step $n+1$).\\ 
\indent Now, let $s_0 > 0$ and, for any
$n\geq 0$, set
\[
\alpha_{n+1} = s_{n} \frac{r_n}{1-r_n}, \qquad
s_{n+1} =  \alpha_{n+1} + s_{n} = \frac{s_n}{1-r_n}
\]
so that
\[
s_{n+1} =
s_0 + \sum_{k=0}^n \alpha_{k+1} =
\frac{s_0}{\prod_{k=0}^{n} (1-r_k)}\,,
\qquad \alpha_{n+1} = s_{n+1}r_n =  
s_0\frac{r_n}{\prod_{k=0}^{n} (1-r_k)}
\]
and, by \eqref{rw-dyn}, 
\[
s_{n+1} \mathcal{W}_{n+1} = s_{n} \mathcal{W}_{n} + \alpha_{n+1} Y_{n+1}. 
\]
Set $H_{n}=s_{n}\mathcal{W}_{n}$
for each $n\geq 0$.  By induction, we get
\[
H_{n+1} = H_{n} + \alpha_{n+1} Y_{n+1}= H_{0} + \sum_{k=0}^n \alpha_{k+1} Y_{k+1}.
\]
If $K_{n} = s_{n} (1 - {\mathcal W}_{n})$ for each $n\geq 0$, then we have
\[
K_{n+1} = K_{n} + \alpha_{n+1} (1 - Y_{n+1}) =
K_{0} + \sum_{k=0}^n \alpha_{k+1} (1 - Y_{k+1}).
\]
(Note that $H_n$ and $K_n$ can be interpreted as the numbers of balls in the urn of 
color $A$ and color $B$, respectively, at time-step $n$).
Moreover $s_{n+1} = H_{n+1} + K_{n+1}$.  \\ \indent Summing up, we
have shown that $\mathcal{W}_n$ corresponds to the proportion
$Z_n=H_n/s_n$ of balls of color $A$ inside the urn at time-step $n$
for a two-color urn process where $S_0=s_0$ is the initial number of balls in the urn, 
the number of balls of color $A$ (resp. $B$) added to the urn at time-step $n$ is $U_n^A=\alpha_{n}Y_n$
(resp. $U_n^B=\alpha_n(1-Y_n)$).  Indeed, \eqref{rw-dyn} and
\eqref{urn-prop-dyn} coincide since $R_{n+1}=
\frac{U_{n+1}^A+U_{n+1}^B}{S_{n+1}}=\frac{\alpha_{n+1}}{s_{n+1}}=r_n$.
\end{proof}

\begin{remark} \rm Note that in the above proposition, 
we only give the number of added balls $U_n^A$ and $U_n^B$ at
  each time-step $n$ in terms of $\alpha_n$ and $Y_n$. We give no
  specifications about the conditional distribution of $Y_{n+1}$ given
  $(\mathcal{W}_0, Y_1,\dots, Y_n)$, that is the
  updating mechanism of the urn. Even if we require that
  $(\mathcal{W}_n)_n$  
  is a martingale with respect to some filtration $\mathcal{G}=({\mathcal G}_n)_n$  
  (as below), that is
  $E[Y_{n+1}\vert \mathcal{G}_n]=\mathcal{W}_n$ a.s., this is not enough to
  determine the conditional distribution of $Y_n$ given
  $\mathcal{G}_n$, except for the trivial case when the random
  variables $Y_n$ are indicator functions.
\end{remark}

Now, let $M=(M_n)_{n\geq 0}$ be a martingale with respect to 
 some filtration $\mathcal G=({\mathcal G}_n)_n$, 
taking values in $[0,1]$ and following the dynamics of \eqref{rw-dyn}, that is
\begin{equation}\label{mart-dyn}
M_{n+1}=(1-r_n)M_n+r_nY_{n+1},\quad n\geq 0,
\end{equation}
where $0\leq r_n<1$, $Y_{n+1}$ takes values in $[0,1]$ and $E[Y_{n+1}\vert {\mathcal
    G}_n]=M_n$. Set $M_\infty\stackrel{a.s.}=\lim_n M_n$.
\\

\indent In the following theorem we will present a sufficient condition 
ensuring that the probability that the process 
$(M_n)_n$ converges to the barrier-set $\{0,1\}$ is zero. 
The merit of this result is that it is 
very general, as it holds for any martingale whose dynamics can be written as in \eqref{mart-dyn}.
\\
\indent Before presenting the theorem, notice that when $P(M_0=0)>0$, 
we trivially have a strictly positive probability of
touching the barrier-set $\{0\}$ in the limit, i.e.~$P(M_\infty=0)>0$, since we obviously have
$P(M_\infty=0\vert M_0=0)=1$.   
On the contrary, when $P(M_0=1)>0$, we trivially have
$P(M_\infty=1)>0$ as $P(M_\infty=1\vert M_0=1)=1$.
For this reason, in the next result the probability of touching the barriers in the limit will be presented given the set  
$\{0<M_0<1\}$.

\begin{theorem}\label{mart-th}
 If $P(0<M_0<1)>0$ and  
\begin{equation}\label{eq-cond-zero2}
r_n = O\Big( e^{-\sum_{k=0}^n r_k} \sum_{k=0}^n r_k \Big),
\end{equation}
then $P(M_\infty=0\vert\,0<M_0<1)=P(M_\infty=1\vert\,0<M_0<1)=0$.
\end{theorem}

In order to prove the stated result, we generalize
the technique used in \cite[Lemma 1]{LamPagTar04}. Firstly, we present 
some auxiliary results that will be proven in Appendix~\ref{app:proofs}.

\begin{lemma}\label{lem-equivalent}
Let $\alpha_n$ be defined as in \eqref{eq:alpha_urn_model} and $s_n = s_0+\sum_{k=1}^n\alpha_k$. We have:
\begin{itemize}
\item[a)] If $\sum_n r_n =+\infty$, we have
\begin{equation}\label{eq-div-serie}
\sum_{n=0}^\infty \frac{r_n}{\sum_{k=0}^n r_k} =+\infty.
\end{equation}
\item[b)] If $0<\sup_n r_n<1$, then we 
have
\footnote{$a_n\asymp b_n$ means that $0<\liminf_n \tfrac{a_n}{b_n}\leq\limsup_n \tfrac{a_n}{b_n}<+\infty$.}.
\begin{equation}\label{eq-equivalence-sumr-logs}
\sum_{k=0}^n r_k =
\sum_{k=0}^n \frac{\alpha_{k+1}}{s_{k+1}} \asymp \ln(s_{n+1})
= 
\ln\Big(\frac{s_0}{\prod_{k=0}^n(1-r_k)}\Big)
\end{equation}
\item[c)] Condition \eqref{eq-cond-zero2} 
implies $\sum_n r_n^2 <+\infty$ and  
is equivalent to 
\begin{equation}\label{eq-cond-zero}
\frac{r_n}{\sum_{k=0}^n r_k} = O\Big(\prod_{k=0}^{n}(1-r_k)\Big)\,,
\end{equation}
which is equivalent to 
\begin{equation}\label{eq-cond-zero3}
\alpha_{n+1} = O(\ln(s_{n+1}))
\end{equation}
and also to 
\begin{equation}\label{eq-cond-zero4}
r_n = O( \ln (s_{n+1}) / s_{n+1} ).
\end{equation} 
\end{itemize}
\end{lemma}

\indent We can interpret $\prod_{k=0}^n(1-r_k)$ as a measure of
the memory of the process at time-step $n$.  Therefore, the above
condition \eqref{eq-cond-zero}, equivalent to \eqref{eq-cond-zero2}, can be read as a bound for 
the amount of new information in order to avoid that it becomes too large with respect to the
historical information observed in the past until time-step
$n$. Then, since the contribution of the past in defining the current
status $M_n$ remains relevant, the process
$M_n$ cannot move too far from the initial values, and
this avoids it to touch the barriers $0$ or $1$ in its limit.

\begin{proof}[Proof of Theorem~\ref{mart-th}] 
Without loss of generality, we can assume $P(0<M_0<1)=1$ (otherwise, it is enough to replace $P$ by $P(\cdot|0<M_0<1)$). 
In this proof we will focus only on the case $\sum_n r_n=+\infty$ as 
when $\sum_n r_n<+\infty$ condition \eqref{eq-cond-zero2} is trivially satisfied and 
we have 
$$
0<M_0\prod_{k=0}^\infty(1-r_k) \leq 
M_\infty \leq
1-(1-M_0)\prod_{k=0}^\infty(1-r_k)<1
\quad \mbox{a.s.}\,.
$$
In the sequel we use the notation of the above Proposition \ref{th-corresp} and 
we split the rest of the proof in some steps.
\\

\noindent{\em First step (proof of $H_n=s_n M_n\stackrel{a.s.}\to +\infty$ 
and $K_n=s_n(1-M_n)\stackrel{a.s.}\to +\infty$):} 
\\
As seen in the proof of Proposition \ref{th-corresp}, we have 
$$
H_0=s_0M_0\quad\mbox{and}\quad H_{n+1}=s_0M_0 + \sum_{k=0}^n \alpha_{k+1} Y_{k+1}.
$$
Therefore, we have  $(H_n)_n$ increasing and 
\begin{equation*}
\begin{split}
\lim_n H_n=\limsup_n H_n&\geq \limsup_n \sum_{k=0}^n \alpha_{k+1} Y_{k+1}=
\limsup_n \sum_{k=0}^n \Big(\sum_{\ell=0}^k r_\ell \Big) \frac{\alpha_{k+1}}{\sum_{\ell=0}^k r_\ell} Y_{k+1}\\
&\geq  r_0 \sum_{k=0}^{+\infty} \frac{\alpha_{k+1}}{\sum_{\ell=0}^k r_\ell} Y_{k+1}.
\end{split}
\end{equation*}
It follows, that $H_n\to +\infty$ a.s.\ if
$\sum_{k} \frac{\alpha_{k+1}}{\sum_{\ell=0}^k r_\ell} Y_{k+1}=+\infty$ a.s. 
For proving this last fact, we recall that, by \eqref{eq-cond-zero}, 
we have 
\[
\frac{\alpha_{k+1}}{\sum_{\ell=0}^k r_\ell}= 
s_0\frac{r_k}{\prod_{\ell=0}^{k} (1-r_\ell)}
\frac{1}{\sum_{\ell=0}^k r_\ell}= O(1). 
\]
Therefore, 
by Theorem~\ref{del-mey-result-app} reported in Appendix \ref{app:proofs},  
$\sum_{k} \frac{\alpha_{k+1}}{\sum_{\ell=0}^k r_\ell} Y_{k+1}=+\infty$ a.s.\
if and only if $\sum_{k} \frac{\alpha_{k+1}}{\sum_{\ell=0}^k r_\ell} M_{k}=+\infty$ a.s. 
But, we have 
$$
\alpha_{k+1} M_k\geq \alpha_{k+1} M_0 \prod_{\ell=0}^k (1-r_\ell)=s_0 M_0 r_k,
$$ 
so that by \eqref{eq-div-serie}
$$
\sum_k \frac{\alpha_{k+1}}{\sum_{\ell=0}^k r_\ell} M_{k}\geq s_0 M_0 
\sum_{k} \frac{r_k}{\sum_{\ell=0}^k r_\ell}= +\infty.
$$
By a symmetric argument, we get $K_n\to +\infty$ a.s.
\\

\noindent{\em Second step (proof of $\liminf_n\ln(s_n)/H_n\stackrel{a.s.}=0$ and $\liminf_n \ln(s_n)/K_n\stackrel{a.s.}=0$):}
We will prove that $\limsup_n \tfrac{H_n}{\ln(s_n)} \stackrel{a.s.}= +\infty$. This is 
equivalent to showing that, for any $a>0$, the event
\[
A_{a} = \Big\{ 
\sup_{n} 
\frac{\sum_{k=0}^{n-1} \alpha_{k+1} Y_{k+1} }{\ln(s_n)} \leq a
\Big\}
\]
has probability zero. 
To this end, fix $b > a$ and let us define, for any $n\in\mathbb{N}$, 
the sets
\[
A_{a,n} = \Big\{ 
\frac{\sum_{k=0}^{n-1} \alpha_{k+1} Y_{k+1} }{\ln(s_n)} \leq a
\Big\}
\quad\text{and}\quad
B_{b,n} = \Big\{ 
\frac{\sum_{k=0}^{n-1} \alpha_{k+1} M_k }{\ln(s_n)} \geq b \Big\}.
\]
To prove $P(A_{a})=0$ we will show the following points:
\begin{itemize}
\item[(i)] $P(A_{a}) = \lim_n P (A_{a} \cap B_{b,n})$;
\item[(ii)] $\lim_n P (A_{a} \cap B_{b,n}) \leq  \limsup_n P(A_{a,n} \cap B_{b,n})$;
\item[(iii)] for any $\epsilon>0$ there exists a sufficiently large $b$ such that $\limsup_n P(A_{a,n} \cap B_{b,n})\leq \epsilon$.
\end{itemize}
Regarding point (i), we know 
from the first step of this proof
that $H_k\rightarrow+\infty$ a.s.,
and so, for any $\lambda>0$, the probability that $ M_k = \tfrac{H_k}{s_k}  \geq \tfrac{\lambda}{s_k}\geq 
\tfrac{\lambda}{s_{k+1}}$ eventually is one.
Therefore, by \eqref{eq-equivalence-sumr-logs}, we have (for a suitable constant $c>0$) that
\[
\liminf_n \frac{\sum_{k=0}^{n-1} \alpha_{k+1} M_k }{\ln(s_n)}\geq
\lambda\liminf_n \frac{\sum_{k=0}^{n-1}\tfrac{\alpha_{k+1}}{s_{k+1}} }{\ln(s_n)}
 \geq c\lambda \qquad a.s.
\]
which implies, by the arbitrariness of $\lambda$,
\[
\liminf_n \frac{\sum_{k=0}^{n-1} \alpha_{k+1} M_k }{\ln(s_n)}  = +\infty \qquad a.s.
\]
Then
$P (\liminf_n B_{b,n}) = 1$, hence $\lim_n P(B_{b,n}) = 1$ and (i) is verified.

For point (ii), it is enough to notice that $A_{a} = \cap_n A_{a,n}$.

Finally, regarding point (iii), set $C=\sup_n\tfrac{\alpha_n}{\ln(s_n)}<+\infty$ by \eqref{eq-cond-zero} 
and  let $n$ be fixed and define for any $k = 0,\ldots,n-1$
\[
X^*_{k+1} = \frac{\alpha_{k+1} }{C\ln(s_n)} Y_{k+1}, \qquad 
M^*_k = E ( X^*_{k+1} |\mathcal{G}_k) = \frac{\alpha_{k+1} }{C\ln(s_n)} M_k, 
\]
while $X^*_{k+1} = M^*_k = 0$ for $k \geq n$.
With this notations, $0 \leq X^*_{k} \leq 1$, $\tau^* \equiv n-1$ is a stopping time, and
$A_{a,n} = \{ \sum_{k=0}^{\tau^*} X^*_{k+1} \leq a/C\}$, 
$B_{b,n} = \{ \sum_{k=0}^{\tau^*} M^*_{k+1} \geq b/C\}$.
Then, by \cite[Theorem~1]{freed73}, we have the following upper bound
\[
P\Big( A_{a,n} \cap B_{b,n}\Big) = P \Big( \sum\nolimits_{k=0}^{\tau^*} X^*_{k+1} \leq \frac{a}{C},
\sum\nolimits_{k=0}^{\tau^*} M^*_{k+1} \geq \frac{b}{C}\Big)  \leq 
\Big(\frac{b}{a}\Big)^{a/C} e^{(a-b)/C},
\]
which holds uniformly in $n$.
This naturally implies that there exists $b$ large enough such that $P( A_{a,n} \cap B_{b,n})<\epsilon$ for any $n\in\mathbb{N}$.
This concludes the proof.\\
By a symmetric argument we can obtain the same limit relation also for $K_n$.
\\

\noindent{\em Third step (conclusion):}\\
We observe that 
\begin{equation*}
\begin{split}
E[(M_\infty-M_n)^2|\mathcal{G}_n]&=\sum_{k\geq n}r_{k}^2 E[(Y_{k+1}-M_k)^2|\mathcal{G}_n] 
=\sum_{k\geq n} r_k^2 E[E[(Y_{k+1}-M_k)^2|\mathcal{G}_k]|\mathcal{G}_n]\\
&\leq\sum_{k\geq n} r_k^2 E[M_k-M_k^2|\mathcal{G}_n]\leq M_n \sum_{k\geq n} r_k^2.
\end{split}
\end{equation*}
Moreover, we have 
$$
(M_\infty-M_n)^2=M_n^2 \mathbbm{1}_{\{M_\infty=0\}}+\mathbbm{1}_{\{M_\infty\neq 0\}} (M_\infty-M_n)^2
\geq M_n^2 \mathbbm{1}_{\{M_\infty=0\}}.
$$
Therefore, recalling that $C=\sup_n\tfrac{\alpha_n}{\ln(s_n)}<+\infty$
and that $x \mapsto \tfrac{\ln(x)}{x^2}$ is decreasing for $x\geq2$, we have, for any $n$ with $s_n\geq 2$,
\begin{equation*}
\begin{split}
M_n^2 P(M_\infty=0|\mathcal{G}_n)&\leq E[(M_\infty-M_n)^2|\mathcal{G}_n]\leq M_n \sum_{k\geq n} r_k^2
=M_n \sum_{k\geq n}\frac{\alpha_{k+1}^2}{s_{k+1}^2}\\
&\leq 
M_n \sum_{k\geq n}C \ln(s_{k+1})\frac{\alpha_{k+1}}{s_{k+1}^2}
=C M_n\sum_{k\geq n}\int_{s_k}^{s_{k+1}}\frac{\ln(s_{k+1})}{s_{k+1}^2}\,dx\\
&\leq C M_n \int_{s_n}^{+\infty} \frac{\ln(x)}{x^2}\,dx
= C M_n \Big( \frac{\ln(s_n)+1}{s_n} \Big),
\end{split}
\end{equation*}
which, taking into account that $s_n\to +\infty$, means 
$M_n^2 P(M_\infty=0|\mathcal{G}_n)=O(M_n\ln(s_n)/s_n)$.
Since $M_n=H_n/s_n$, it follows $P(M_\infty=0|\mathcal{G}_n)=O(\ln(s_n)/H_n)$ and so 
(denoting by $c$ a suitable constant) we get 
$$
\mathbbm{1}_{\{M_\infty=0\}}\stackrel{a.s.}=\lim_n E[\mathbbm{1}_{\{M_\infty=0\}}|\mathcal{G}_n]=
\lim_n P(M_\infty=0|\mathcal{G}_n)\leq
c\liminf_n \frac{\ln(s_n)}{H_n} \stackrel{a.s.}=0,
$$
that is $P(M_\infty=0)=0$.\\
By a symmetric argument, we obtain $P(M_\infty=1)=0$.
\end{proof}

\indent We will now present the conditions to ensure that 
the probability of touching the barrier-set $\{0,1\}$ in the limit for the general class of martingales $(M_n)_n$ 
with dynamics \eqref{urn-prop-dyn} is strictly positive. 
In particular, we will focus on the barrier $\{0\}$, as the results for the barrier $\{1\}$ 
would be completely analogous. 
Moreover, the result will be presented conditioning on the set $\{0<M_0<1\}$ since, as already observed, 
it is trivial that $P(M_\infty=0\vert M_0=0)=1$
and $P(M_\infty=0\vert M_0=1)=0$.\\
\indent Before stating the result, let us first present the required assumptions and a technical remark.

\begin{assumption}\label{ass:2}
Assume $P(0<M_0<1)>0$ and that there exist:
\begin{enumerate}
\item\label{ass:2a} a sequence $(\delta_n)_n$ in $[0,1]$ such that
\(
\sum_{n=0}^{+\infty} \delta_n r_n = +\infty
\);
\item\label{ass:2b} a sequence of non-decreasing functions $g_n: [0,1] \to [0,1]$ 
such that, for any $C>0$,
\[
\sum_{n=1}^{+\infty} g_n\Big( \min\Big(C \prod_{k=0}^{n-1} (1 - \delta_k r_k) , 1 \Big) \Big) < +\infty;
\]
\item\label{ass:2c} a reinforcement mechanism based on 
$\{n_0, (\delta_n,g_n)_n\}$, with $n_0\in\mathbb{N}$:
setting
\[
\begin{aligned}
& A_{n+1} = \{Y_{n+1} \leq (1-\delta_n) M_n\},\\
& B_{n} = \{P(A_{n+1} | \mathcal{G}_n) \geq 1-g_n(M_n) \},\\
& C_{n_0}=
\cap_{n\geq n_0}\{A_n^c\cup B_n\} 
=
\cap_{n\geq n_0}\{\mathbbm{1}_{A_n}\leq \mathbbm{1}_{B_n}\}
,
\end{aligned}
\]
there exists and event $E_{n_0}\in \mathcal{G}_{n_0}$ such that
$E_{n_0}\subseteq C_{n_0}$ and $P(E_{n_0} \cap A_{n_0+m_0} |\, 0<M_0<1)>0$ 
for some $m_0\geq 1$. 
\end{enumerate}
\end{assumption}

\begin{remark}
Let $P_{E_{n_0}}(\cdot) = P(\cdot | E_{n_0})$ and 
$E_{E_{n_0}}[\cdot | \mathcal{G}] = E_{P_{E_{n_0}}}[\cdot | \mathcal{G}]$, that is the conditional expectation 
given the $\sigma$-field  $\mathcal{G}$ with respect to $P_{E_{n_0}}$. Then, for any non-negative random variable $X$ 
and $n \geq n_0$, we have 
\begin{equation}\label{eq:P_PE0}
E[X|\mathcal{G}_n]  \geq E[X|\mathcal{G}_n] \mathbbm{1}_{E_{n_0}} = E_{E_{n_0}}[X|\mathcal{G}_n]\mathbbm{1}_{E_{n_0}},
\qquad \text{a.s.\ (and $P_{E_{n_0}}$-a.s.)}
\end{equation}
Let $X = \mathbbm{1}_{A_{n+1}}$. As a consequence of Assumption~\ref{ass:2}.\ref{ass:2c}, we have for $n\geq n_0$ 
\begin{equation}\label{eq:PE0_reinf}
P_{E_{n_0}}(A_{n+1} |\mathcal{G}_n) \mathbbm{1}_{A_{n}} 
=P(A_{n+1} |\mathcal{G}_n) \mathbbm{1}_{A_{n}} 
\geq  (1 - g_n(M_n)) \mathbbm{1}_{A_{n}},
\quad \text{$P_{E_{n_0}}$-a.s.}
\end{equation}
\end{remark}

\begin{example}
For simplicity, assume that $P(0<M_0<1)=1$ (otherwise, replace $P$ by $P(\cdot|0<M_0<1)$). 
Now, consider the case when there exist $n_0\in\mathbb{N}$ and an event $E_{n_0}$, observable at time-step $n_0$ 
(this is the meaning of the $\mathcal{G}_{n_0}$-measurability) 
such that $P(E_{n_0})>0$ and, given the occurrence of $E_{n_0}$,  
$Y_{n+1}$ takes value in $\{0,1\}$ for each $n\geq n_0$. 
Then, there exists a reinforcement mechanism based on $\{n_0, (\delta_n\equiv 1,g_n\equiv\text{id})_n\}$. Indeed, 
$A_{n+1} = \{Y_{n+1} = 0\}$ and, by \eqref{eq:P_PE0}, we have 
$P(Y_{n+1}=0|\mathcal{G}_n) \mathbbm{1}_{E_{n_0}}=P_{E_{n_0}}(Y_{n+1}=0|\mathcal{G}_n) \mathbbm{1}_{E_{n_0}}=
(1-M_n)\mathbbm{1}_{E_{n_0}}=(1-g_n(M_n))\mathbbm{1}_{E_{n_0}}$ a.s. for each $n\geq n_0$. 
Therefore,   we have $E_{n_0}\subseteq B_n$ (a.s.) for any $n\geq n_0$ and 
so $E_{n_0}\subseteq C_{n_0}$ (a.s.). Assumption \ref{ass:2}.\ref{ass:2c} 
is equivalent to requiring  that  
$P_{E_{n_0}}(Y_{n_0+m_0}=0)=1-E[M_{n_0+m_0-1}|E_{n_0}]>0$ for some $m_0\geq 1$. 
But, this is true for each $m_0\geq 1$ since, by the martingale property, we get 
$E[M_{n_0+m_0-1}|E_{n_0}]=E[M_{n_0}|E_{n_0}]$ and this last quantity belongs to $(0,1)$ 
since $r_n<1$ and $Y_n\leq 1$ for each $n$ in the dynamics \eqref{mart-dyn} 
and so $M_{n_0}<1$ a.s. on $\{0<M_0<1\}$. 
Assumption \ref{ass:2}.\ref{ass:2a} reduces to $\sum_n r_n = +\infty$, that we already know to be necessary for
$P(M_{\infty} = 0|\, 0<M_0<1)>0$, while
Assumption \ref{ass:2}.\ref{ass:2b} reads as $\sum_{n=0}^\infty \prod_{k=0}^n (1 - r_k)  < +\infty$.
\end{example}

\begin{theorem}
  \label{th2}
  Let Assumption~\ref{ass:2} hold. Then 
  $$P(M_{\infty} = 0|\, 0<M_0<1)\geq P(\cap_{n=n_0+m_0}^{+\infty} A_n|0<M_0<1) > 0.$$
Specifically, on $\{0<M_0<1\}$, for each $m\geq 1$, we have 
\footnote{where $\prod_{k=n_0+m}^{n_0+m-1}=1$ by convention.}
\begin{multline}\label{eq:prob_fixation_0}
P(M_{\infty} = 0 | \mathcal{G}_{n_0+m})
\geq P( \cap_{n=n_0+m}^{+\infty} A_{n} \big| \mathcal{G}_{n_0+m} )\geq \\
\mathbbm{1}_{E_{n_0}\cap A_{n_0+m}}
\prod_{n=n_0+m}^{+\infty} \Big( 1 - g_{n} \Big( M_{n_0+m} \prod_{k=n_0+m}^{n-1} \big( 1 - \delta_{k} r_{k}\big)\Big)\Big)
\qquad a.s.\,.
\end{multline}
\end{theorem}

\begin{remark}\label{rem-for-1}
If Assumption \ref{ass:2} holds for $(1-M_n)$ and $(1-Y_n)$, then $P(M_{\infty} = 1|\, 0<M_0<1) > 0$ and an inequality analogous to $\eqref{eq:prob_fixation_0}$ holds true.
\end{remark}

Before presenting the proof of Theorem~\ref{th2}, let us briefly discuss the basic ideas 
behind it and the conditions reported in Assumption~\ref{ass:2}.
\\
\indent The proof of $P(M_\infty=0|\, 0<M_0<1)>0$ will be realized by extending the definition of 
\emph{fixation}, which typically refers to 
the event that a real process definitively assume the same fixed value, e.g. in this case $\cap_{n\geq n_0}\{Y_{n}=0\}$.
In our general framework, where the distribution of $Y_{n+1}$ given $\mathcal{G}_n$ is not specified, and so not necessarily discrete, 
we introduce the notion of  ``$\delta_n$-\emph{fixation}'' as the fixation on the value $1$ 
of the sequence of $\{\mathbbm{1}_{A_n}:n\geq n_0+m_0\}$ and 
we prove that it occurs with strictly positive probability given $\{0<M_0<1\}$, i.e. $P(\cap_{n\geq n_0+m_0} A_n|0<M_0<1)>0$. 
Assumption \ref{ass:2}.\ref{ass:2a} ensures that
this $\delta_n$-\emph{fixation} implies $\{M_\infty=0\}$.
Indeed, since on each $A_{n+1}$ we have
$M_{n+1}\leq (1-\delta_n r_n)M_n$,
condition $\sum_n \delta_n r_n  = +\infty$
ensures that the decrease of the process $M_n$ on the fixation event
is strong enough to reach the barrier $\{0\}$. Assumption \ref{ass:2}.\ref{ass:2c}
ensures $A_n\cap E_{n_0}\subseteq B_n\cap E_{n_0}$ for any $n\geq n_0$, where $E_{n_0}$ is an event with 
strictly positive probability, given $\{0<M_0<1\}$, and observable at time-step $n_0$ (i.e. $E_{n_0}\in\mathcal{G}_{n_0}$). 
Since each $B_n$ provides a lower bound on the probability of the next set $A_{n+1}$, 
we can read Assumption \ref{ass:2}.\ref{ass:2c} as the existence of a 
\emph{triggering mechanism}, with a strictly positive probability of starting at time $m_0$ (i.e. $P(E_{n_0} \cap 
A_{n_0 +m_0} | 0 < M_0 < 1) > 0$ for some $m_0 \geq 1$), for the sequence of sets 
$\{A_n: n\geq n_0+m_0\}$: the occurrence of $A_n$ implies, with at least a certain probability, the occurrence of $A_{n+1}$, 
which, iterating this argument for any $n\geq n_0+m_0$, is the key point of the $\delta_n$-\emph{fixation}. 
Then, Assumption \ref{ass:2}.\ref{ass:2b} ensures that $g_{n}(\pi_{n})\rightarrow 0$ sufficiently fast to have
$\sum_{n}g_{n}(\pi_{n})<+\infty$, for any $\pi_{n} = O (\prod_{k=0}^n (1 - \delta_k r_k))$. This fact guarantees
that the above triggering mechanism implies the $\delta_n$-\emph{fixation} with a strictly positive probability.

\begin{proof}[Proof of Theorem \ref{th2}]
Without loss of generality, we can assume $P(0<M_0<1)=1$ (otherwise, it is enough to replace $P$ by $P(\cdot|0<M_0<1)$). First note that, on $A_{n+1}$, $M_{n+1} \leq (1-\delta_nr_n) M_n$, so that, for each $m\geq 1$ and $N\geq 1$, on $ \cap_{k=n_0+m}^{n_0+m+N-1} A_{k+1} $, we have 
\begin{equation}\label{eq:Mn2zero}
M_{n_0+m+N} \leq M_{n_0+m} 
\prod_{k=n_0+m}^{n_0+m+N-1} \big( 1 - \delta_{k} r_{k}\big).
\end{equation}
Hence, because of Assumption~\ref{ass:2}.\ref{ass:2a}, we have 
$\cap_{k=n_0+m}^{+\infty} A_{k} \subseteq \{M_{\infty} = 0\}$. 
Moreover, by Assumption~\ref{ass:2}.\ref{ass:2c}, \eqref{eq:PE0_reinf} and \eqref{eq:Mn2zero},
we obtain, $P_{E_{n_0}}$-a.s.
\begin{multline*}
P_{E_{n_0}}\Big( \bigcap_{n=n_0+m}^{n_0+m+N+1} A_{n} \Big| \mathcal{G}_{n_0+m} \Big)
\\
\begin{aligned}
& = E_{E_{n_0}}\Big[ E_{E_{n_0}} [ \mathbbm{1}_{A_{n_0+m+N+1}} | \mathcal{G}_{n_0+m+N}]  \mathbbm{1}_{A_{n_0+m+N}} 
\prod_{n=n_0+m}^{n_0+m+N-1} \mathbbm{1}_{A_{n}}  \Big| \mathcal{G}_{n_0+m} \Big]
\\
& \geq E_{E_{n_0}}\Big[ \big(1 - g_{n_0+m+N} (M_{n_0+m+N}) \big) 
\prod_{n=n_0+m}^{n_0+m+N} \mathbbm{1}_{A_{n}}  
\Big| \mathcal{G}_{n_0+m} \Big]\\
 &\geq \Big(1 - g_{n_0+m+N} \Big( M_{n_0+m} 
\prod_{k=n_0+m}^{n_0+m+N-1} \big( 1 - \delta_{k} r_{k}\big)\Big)\Big)
P_{E_{n_0}}\Big(
 \bigcap_{n=n_0+m}^{n_0+m+N} A_{n}  \Big| \mathcal{G}_{n_0+m} \Big),
\end{aligned}
\end{multline*}

which leads by induction to
\begin{multline*}
P_{E_{n_0}} \big( \cap_{n=n_0+m}^{n_0+m+N+1} A_{n} \big| \mathcal{G}_{n_0+m} \big) 
\geq\\
\mathbbm{1}_{A_{n_0+m}}
\prod_{n=n_0+m}^{n_0+m+N} 
\Big(1 - g_{n} \Big( M_{n_0+m} 
\prod_{k=n_0+m}^{n-1} \big( 1 - \delta_{k} r_{k}\big)\Big)\Big)\qquad P_{E_{n_0}}-a.s.,
\end{multline*}
that gives \eqref{eq:prob_fixation_0} by \eqref{eq:P_PE0} and letting $N\rightarrow+\infty$ and recalling 
that we have $\cap_{k=n_0+m}^{+\infty} A_{k} \subseteq \{M_{\infty} = 0\}$. 
Finally, let $m=m_0$ as in Assumption~\ref{ass:2}.\ref{ass:2c} such that
$P_{E_{n_0}} (A_{n_0+m_0})>0$.
From \eqref{eq:prob_fixation_0}, 
taking the mean value of both sides, we get that
$P( M_\infty=0)>0$ by 
Assumption~\ref{ass:2}.\ref{ass:2b} as follows
\begin{multline*}
P\big( M_\infty=0 \big) = 
E[P\big( M_\infty=0 \big| \mathcal{G}_{n_0+m_0} \big)]
\geq P( \cap_{n=n_0+m_0}^{+\infty} A_{n} \big| \mathcal{G}_{n_0+m_0} )
\\
\begin{aligned}
&\geq
E\Big[\mathbbm{1}_{E_{n_0}\cap A_{n_0+m_0}}
\prod_{n=n_0+m_0}^{+\infty} \Big( 1 - g_{n} \Big( M_{n_0+m_0} \prod_{k=n_0+m_0}^{n-1} \big( 1 - \delta_{k} r_{k}\big)\Big)\Big)\Big]
\\
&=
E\Big[\mathbbm{1}_{E_{n_0}\cap A_{n_0+m_0}}
\prod_{n=n_0+m_0}^{+\infty} \Big( 1 - g_{n} \Big( \tfrac{M_{n_0+m_0}}{\prod_{k=0}^{n_0+m_0-1} \big( 1 - \delta_{k} r_{k}\big)} \prod_{k=0}^{n-1} \big( 1 - \delta_{k} r_{k}\big)\Big)\Big)\Big]
\\
&\geq
P(E_{n_0} \cap A_{n_0+m_0} )
\prod_{n=n_0+m_0}^{+\infty} \Big( 1 - g_n\Big( \min\big(C \prod_{k=0}^{n-1} (1 - \delta_k r_k) , 1 \big) \Big)  \Big)
\\
&>0,
\end{aligned}
\end{multline*}
where $C^{-1} = \prod_{k=0}^{n_0+m-1} ( 1 - \delta_{k} r_{k})$ and, using the fact that $g_n(\cdot)$ is a non-decreasing function,  
we have replaced $M_{n_0+m_0}\leq 1$ by $1$.
\end{proof}

The next result deals with the case in which the probability of touching the barriers in the limit is not only positive 
as in Theorem \ref{th2}, but it is exactly equal to one.

\begin{theorem}\label{mart-th-bis}
Set
$ L_\infty = \liminf_{n\rightarrow\infty} Var[Y_{n+1}|\mathcal{G}_n]$ and assume
 \begin{equation}\label{eq:condition_liminf_variance}
P (\{M_\infty(1-M_\infty)>0\}\cap
 \{L_\infty = 0\} ) = 0.
 \end{equation}
Then, if $\sum_n r_n^2=+\infty$, 
we have $P(M_\infty=0)+P(M_\infty=1)=1$  and $P(M_\infty=1)=E[M_0]$. 
\end{theorem}

Condition \eqref{eq:condition_liminf_variance} is a natural assumption. 
It ensures that, asymptotically, the variance of the reinforcement variables $(Y_{n})_n$ is bounded away from zero 
whenever $M_\infty\in(0,1)$. 
If this is not the case, the convergence to the barriers may not be related  
with the type of reinforcement sequence $(r_n)_n$. 
For instance, when $Y_{n+1} = M_{n} $ (that means $ Var[Y_{n+1}|\mathcal{G}_n]=0$) eventually, then $M_{n}$ 
is definitively constant (not equal to $0$ or $1$ when $0<M_0<1$) whatever the sequence $(r_n)_n$ is. 

\begin{proof}[Proof of Theorem \ref{mart-th-bis}]
Let us first denote by $\langle M\rangle=(\langle
  M\rangle_n)_n$ the predictable compensator of the
  submartingale $M^2=(M_n^2)_n$.  Since
  $M$ is a bounded martingale, we have that $ \langle
  M\rangle_n$ converges a.s.\ and its limit $\langle
  M\rangle_\infty$ is such that
  $E[\langle M\rangle_\infty]<+\infty$
  (and so $\langle M\rangle_\infty<+\infty$ a.s.).  Then,
we observe that
\begin{equation*}
     \langle M\rangle_\infty =
 \sum_n ( \langle M\rangle_{n+1}-\langle M\rangle_{n})
 = \sum_n E[( M_{n+1}-M_{n})^2\vert \mathcal{G}_n]
= \sum_n  r^2_n Var[Y_{n+1}\vert \mathcal{G}_n]\,.
  \end{equation*}
Therefore, for each $\epsilon>0$, the event 
 $$
 A_\epsilon= \bigcup_n\bigcap_{k\geq n}\{Var[Y_{k+1}|\mathcal{G}_k]\geq L_\infty-\epsilon\}
 $$
 is contained (up to a negligible set) in the event $\{(L_\infty-\epsilon)\sum_n r_n^2<+\infty\}$.
 Since $\sum r_n^2=+\infty$, this last event coincides with $\{L_\infty\leq \epsilon\}$.
 It follows that, since $P(A_\epsilon)=1$ for each $\epsilon>0$ by the definition of $L_\infty$,  we have 
 $P(L_\infty\leq \epsilon)=1$ for each $\epsilon>0$, from which we get 
 $P(L_\infty=0)=1$ and so, by \eqref{eq:condition_liminf_variance},
  $P(M_\infty(1-M_\infty)>0)=0$, that is $P(M_\infty(1-M_\infty)=0)=1$. 
This concludes the proof of the first statement. For the 
last statement, it is enough to note that  $P(M_\infty=1)=E[M_\infty]=E[M_0]$.
\end{proof}

Now, we conclude the picture with a simple general result.

\begin{prop}\label{mart-th-base}
If $P(0<M_0<1)>0$ and $\sum_n r_n^2<+\infty$, then
  $P(M_\infty=0)+P(M_\infty=1)<1$.  
\end{prop}

\begin{proof}
We note that $P(M_\infty=0)+P(M_\infty=1)=1$ if and only if
  $E[M_\infty(1-M_\infty)]=0$. 
Therefore, we set $x_{n}=E[M_{n}(1-M_n)]$ so that
$E[M_\infty(1-M_\infty)]=\lim_n x_n$ and we observe that
\begin{equation*}
  x_{n+1}=E[M_{n+1}]-E[M_{n+1}^2]=E[M_n]-(1-r_n^2)E[M_n^2]-r_n^2E[Y_{n+1}^2].
\end{equation*}
From this equality, we get
$x_{n+1}=(1-r_n^2)x_n+r_n^2\left(E[M_n]-E[Y_{n+1}^2]\right) \geq
(1-r_n^2)x_n$, because $Y_{n+1}$ takes values in $[0,1]$ and so
$E[Y_{n+1}^2]\leq E[Y_{n+1}]=E[M_n]$.  Thus, we have $x_{n+1}\geq
x_0\prod_{k=0}^n (1-r_k^2)$ for each $n$ and so
$E[M_\infty(1-M_\infty)]\geq x_0\prod_{k=0}^{+\infty}(1-r_k^2)$, where
the infinite product is strictly positive when $\sum_n r_n^2<+\infty$.
\end{proof}

Finally, the following remark can be useful in order to describe the distribution of 
the limit random variable $M_\infty$ in the open interval $(0,1)$.

\begin{remark}\label{rem-TLC-mart}
\rm Arguing exactly as in \cite[Theorem~4.2]{ale-cri-ghi}, we get
$$ 
a_n(M_n-M_\infty)\stackrel{stably}\longrightarrow 
{\mathcal N}(0,\Phi M_\infty(1-M_\infty)),  
$$
(for the definition of the stable convergence, see, for instance, 
Appendix~B in~\cite{ale-cri-ghi} and references therein)
 provided that
$E[(Y_{n+1}-M_n)^2|\mathcal{G}_n]\stackrel{a.s.}\to \Phi M_\infty(1-M_\infty)$,
where $\Phi$ is a suitable bounded positive random variable, which is measurable with respect to 
${\mathcal G}_\infty=\bigvee_n {\mathcal G}_n$,
  and that there exists a sequence
$(a_n)_n$ of positive numbers such that $a_n\to +\infty$
$$
a_n^2 \sum_{k\geq n} r_k^2\longrightarrow 1
\qquad\mbox{and}\qquad
a_n\sup_{k\geq n} r_k\longrightarrow 0.
$$ The above convergence is also in the sense of the almost sure
conditional convergence (see \cite[Definition 2.1]{Cri}) with
respect to $(\mathcal{G}_n)_n$. If $P(\Phi>0)=1$, 
this last fact implies that
$P(M_\infty=z)=0$ for all $z\in (0,1)$ (see the proof of 
\cite[Theorem 2.5]{cri-dai-lou-min}). 
\end{remark}


\section{Probability of asymptotic polarization for a network of reinforced stochastic processes}
\label{sec:barriers}

We consider a system of
$N\geq 2$ RSPs with a network-based interaction as defined in Section~\ref{intro}.  
Assuming to be in the scenario of
complete almost sure asymptotic synchronization of the system, 
that is when all the stochastic processes $(Z_{n,l})_n$, with $l\in V$, converge almost surely 
toward the same random variable $Z_\infty$ 
(or, in other terms, when $\vZ_n\stackrel{a.s.}\longrightarrow Z_\infty\vone$),   
we are going to describe the
phenomenon of {\em asymptotic polarization}, i.e. 
to determine if the
common limit variable $Z_\infty$ can or cannot belong to the
barrier-set $\{0,1\}$. To this end, in order to
exclude trivial cases, we fix $P(T_0)<1$ with $T_0=\{\vZ_0=\vzero\}\cup\{\vZ_0=\vone\}$ 
and   we collect the conditions that
ensure the complete almost sure asymptotic synchronization of the
system (see \cite[Corollary 2.5]{ale-cri-ghi-complete}) 
in the following assumption:
\begin{assumption}\label{ass:synchro}
Assume that at least one of the following conditions holds true: 
\begin{enumerate}
\item $\sum_n r_n=+\infty$ and $W$ aperiodic,
\item $\sum_n r_n(1-r_n)=+\infty$.
\end{enumerate}
\end{assumption}
(For the definition of periodicity of a matrix please refer to Appendix~\ref{app:periodic_matrix}).
\noindent Under these conditions, the almost sure random limit
$Z_\infty$ of the system is well defined and we can introduce the set
\begin{equation*}
  T_{\infty} =
  \{{\vZ}_n\stackrel{a.s.}\longrightarrow \vone\} \cup
  \{{\vZ}_n\stackrel{a.s.}\longrightarrow \vzero\} =
  \{Z_\infty=1\} \cup \{Z_\infty=0\}.
\end{equation*}
The rest of this section is dedicated to characterize when we have
$P(T_{\infty}\vert T_0^c)=0$ ({\em non-trivial asymptotic polarization is
  negligible}), $0<P(T_{\infty})<1$ ({\em asymptotic polarization with
  a strictly positive probability, but non almost sure}) or
$P(T_{\infty})=1$ ({\em almost sure asymptotic polarization}). 
In particular, for the second case, we will give a condition that 
assures $P(T_\infty|T_0^c)>0$ 
({\em non-trivial asymptotic polarization with a strictly positive probability}).
\\

\indent Before stating the results, we point out that the conditions reported in
Assumption~\ref{ass:synchro} are essential only for the complete almost sure asymptotic synchronization, 
and so for the existence of $Z_\infty$.  Indeed, the provided results
could be also stated for the case $N=1$ omitting
Assumption~\ref{ass:synchro}.  
\\

\indent We highlight that the key-point for the following results is 
that, in the case of complete almost sure asymptotic synchronization,  
the random variable $Z_\infty$ can be seen as the almost sure limit of the martingale 
$$
\widetilde{Z}_n=\vv^\top\vZ_n\,,
$$
where $\vv$ is the (unique) left eigenvector associated to the leading eigenvalue $\vone$ of 
$W^\top$ with all the entries in $(0, +\infty)$ and such that $\vv^\top \vone = 1$.  Indeed, we note that 
the assumption $P(T_0)<1$ corresponds to $P(0<\widetilde{Z}_0<1)>0$ and 
 the stochastic process $(\widetilde{Z}_n)_n$ is a martingale, 
 with respect to the filtration ${\mathcal F}=(\mathcal{F}_n)_n$ associated to the model (see Section~\ref{intro}), which   
 follows the dynamics
\begin{equation}\label{Z-tilde-dynamics}
\widetilde{Z}_{n+1}=(1-r_n)\widetilde{Z}_n+r_n Y_{n+1},\quad\mbox{with } Y_{n+1}=\vv^\top\vX_{n+1}.
\end{equation} 

The first result follows from Theorem \ref{mart-th} and Lemma \ref{lem-equivalent}, and 
it shows sufficient conditions to guarantee that the
probability of touching the barriers in the limit is zero on $T_0^c$,
i.e. on the set where the initial condition is non-trivial.

\begin{theorem}[Non-trivial asymptotic polarization is
  negligible]\label{th1} Under Assumption~\ref{ass:synchro}, if we also have 
  condition \eqref{eq-cond-zero2} or,
  equivalently, \eqref{eq-cond-zero},
then $P(T_\infty\vert T_0^c)=0$.
In particular, these conditions are satisfied when  there exists  
$0<c\leq 1$ such that 
$\lim_n n r_n=c$ and $\sum_n (r_n-cn^{-1})$ is convergent. 
\end{theorem}

\begin{proof}
The first statement follows by an application of 
Theorem \ref{th1} to $(\widetilde{Z}_n)_n$. 
Regarding the last statement, we note that  
$\sum_{k=0}^n r_k = c\ln(n) + O(1)$ and then 
$e^{-\sum_{k=1}^n r_k }
\sum_{k=1}^n r_k = O(n^{-c}\ln(n))$.
\end{proof}

Note that, since $\sum_n r_n=+\infty$ (by Assumption~\ref{ass:synchro}),  
conditions \eqref{eq-cond-zero2}, or \eqref{eq-cond-zero}, imply $\sum_n\prod_{k=0}^{n}(1-r_k) = +\infty$ 
  (see \eqref{eq-div-serie} in Lemma \ref{lem-equivalent}). In the next result we consider the opposite condition.

\begin{theorem}[Non-trivial asymptotic polarization with a strictly positive probability]
  \label{th2-rsp}
Under Assumption~\ref{ass:synchro}, if
  \begin{equation*}
\sum_n \prod_{k=0}^n (1-r_k)<+\infty,
  \end{equation*}
then we have $P(Z_\infty=z\vert T_0^c)>0$ for both $z=0$ and $z=1$ and so 
$P(T_\infty\vert T_0^c)>0$. More precisely, we have a strictly positive probability, given $T_0^c$, 
of fixation of all the stochastic processes $\{(X_{n,l})_n : l \in V \}$ on the value $z$ for
both $z = 0$ and $z = 1$. \\ 
\indent In particular, the above
assumptions are satisfied when $\lim_n n^\gamma r_n=c$ and $\sum_n (r_n - c n^{-\gamma})$ is convergent   
for $\gamma=1$ and $c>1$ or for  $1/2<\gamma< 1$ and $c>0$.
\end{theorem}

\begin{proof}
This result follows from Theorem \ref{th2} applied to $\widetilde{Z}_n$ and to $(1-\widetilde{Z}_n)$ (see Remark \ref{rem-for-1})
with $n_0=0$, $\delta_n\equiv 1$, $\mathcal{G}_n=\mathcal{F}_n$, $g_n(x)=g(x)=\min\big(\tfrac{x}{v_{\min}},1\big)$ $\forall n$,
where $v_{\min}=\min_l v_l>0$, $E_{n_0}=C_{n_0}$. 
Indeed, we have $A_{n+1}=\{Y_{n+1}=0\}=\cap_{l=1}^N
\{X_{n+1,l}=0\}$ and so a.s.
\[
\begin{aligned}
P(A_{n+1}^c\vert\, \mathcal{G}_n)&=P(\cup_{l=1}^N\{X_{n+1,l}=1\}\vert\, \mathcal{F}_n) 
\\ 
&\leq
\sum_{l=1}^N[W^\top \vZ_n]_l
= \vone^\top W^\top \vZ_n 
\\ 
&\leq
\frac{1}{v_{\min}}\vv^\top W^\top \vZ_n =
\frac{1}{v_{\min}}\vv^\top \vZ_n =
\frac{1}{v_{\min}}\widetilde{Z}_n.
\end{aligned}
\]
It follows $P(A_{n+1}\vert\,\mathcal{G}_n)\geq 1-g(\widetilde{Z}_n)$ a.s. for each $n$.  
This means $P(B_n)=1$ for each $n$, $C_0\in\mathcal{G}_0$ and $P(C_0)=1$.
Assumption \ref{ass:2}.\ref{ass:2b} is simply to be verified as, for any $C>0$,
\begin{align*}
\sum_{n=1}^\infty g_n\Big( \min\Big(C \prod_{k=0}^{n-1} (1 - \delta_k r_k) , 1 \Big) \Big) 
& =
\sum_{n=1}^\infty \min\Big( \frac{C \prod_{k=0}^{n-1} (1 -  r_k)}{v_{\min}} , 1 \Big) 
\\
& \leq
\frac{C}{v_{\min}}
\sum_{n=1}^\infty \prod_{k=0}^{n-1} (1 -  r_k) 
\\
& =
\frac{C}{v_{\min}}
\sum_{n=0}^\infty \prod_{k=0}^n (1 -  r_k) 
\\
&< +\infty.
\end{align*}
Let us now focus on Assumption \ref{ass:2}.\ref{ass:2c} (with $n_0=0$ and $E_{n_0}=C_{n_0}$), 
and in particular on proving that, for some 
$m_0\geq 1$, $P(A_{n_0+m_0}|T_0^c)>0$ holds true.
To this end, we observe that, by the irreducibility of $W$ and the assumption
$P(T_0)< 1$, there exists a time-step $m_*$ such that, the event $\{0 < Z_{m_*,l} < 1 : \forall l\in V\}$ 
has a strictly positive probability under $P(\cdot |T_0^c)$. 
Then, we have
\begin{align*}
P(A_{m_*+1}|T_0^c) 
&=
E[\cap_{l=1}^N
\{X_{m_*+1,l}=0\}|\, T_0^c]
=
 E\Big[\prod_{l=1}^N(1-Z_{m_*,l})|\, T_0^c\Big]
 \\
 &\geq
 E\Big[(1-\max_{l\in V} Z_{m_*,l})^N|\, T_0^c\Big]>0.
\end{align*}
Therefore Assumption \ref{ass:2}.\ref{ass:2c} holds true with $m_0=m_*+1$. 
Hence, we can apply Theorem \ref{th2} to $\widetilde{Z}_n$ and  
we get 
$$
P(Z_\infty = 0\vert\, T_0^c) \geq P(\cap_{n=n_0+m_0}^{+\infty} A_n|T_0^c) > 0\,,
$$
where $\cap_{n=n_0+m_0}^{+\infty} A_n=\cap_{n\geq n_0+m_0,\,l=1,\dots,N}
\{X_{n+1,l}=0\}$.\\
By simmetry $(1-\widetilde{Z}_n)$ and $(1-Y_{n})$
also satisfy Assumption \ref{ass:2} (with the same 
$n_0$, $\delta_n$, $g_n$, $E_{n_0}$ and $m_0$) and so we get
$$P(Z_\infty = 1\vert\, T_0^c)
\geq P(\cap_{n\geq n_0+m_0,\,l=1,\dots,N}
\{X_{n+1,l}=1\}|T_0^c) >0\,.
$$
\indent The last statement of Theorem~\ref{th2-rsp} follows from Lemma~\ref{lemma-product}: indeed,
when $\gamma=1$ and $c>0$, we have $\prod_{k=0}^n(1-r_k) = O(n^{-c})$ and the series
$\sum_n n^{-c}$ is convergent when $c>1$; while, when $1/2<\gamma<1$ and $c>0$, we have 
$\prod_{k=0}^n(1-r_k) = O(\exp( - C n^{1-\gamma}))$ (with $C=c/(1-\gamma)$) and the series 
$\sum_n \exp(-Cn^{1-\gamma})$ is always convergent.
\end{proof}

The last result follow from Theorem \ref{mart-th-bis} and Proposition \ref{mart-th-base}, and 
it affirms that the probability of asymptotic polarization is strictly smaller than or equal to $1$ 
according to the convergence or not of the series $\sum_n r_n^2$.

\begin{theorem}\label{th3}
Under Assumption~\ref{ass:synchro}, we have:
  \begin{itemize}
\item[(i)] If $\sum_n r_n^2<+\infty$, then
  $P(T_\infty)<1$ (non-almost sure asymptotic polarization);
\item[(ii)] If $\sum_n r_n^2=+\infty$, then $P(T_\infty)=1$ 
 (almost sure asymptotic polarization) and 
$P(Z_\infty=1)=\boldsymbol{v}^\top E[\boldsymbol{Z}_0]$.
\end{itemize}
In particular, when $\lim_n n^{\gamma}r_n= c>0$, case
i) is verified if $1/2<\gamma\leq 1$ and case ii) is verified when
$0<\gamma\leq 1/2$.
\end{theorem}

\begin{proof}
Statement (i) follows from Proposition \ref{mart-th-base}.\\
Statement (ii) follows from Theorem \ref{mart-th-bis}, with $\mathcal{G}_n=\mathcal{F}_n$,
since
$$Var[Y_{n+1}|\mathcal{F}_n]=\sum_{l=1}^N v_l^2 [W^\top \vZ_n]_l(1-[W^\top \vZ_n]_l).$$
Then, we have 
$$
Var[Y_{n+1}|\mathcal{F}_n] \stackrel{a.s.}\longrightarrow L_\infty =
\left(\sum_{l=1}^N v_l^2\right) Z_\infty(1-Z_\infty),
$$
and so \eqref{eq:condition_liminf_variance} 
is trivially satisfied.\\

Regarding the last statement of Theorem \ref{th3}, 
we note that $r_n\to 0$ and, since $r_n n^\gamma/c\to 1$, 
the series $\sum_n r_n$ has the same behaviour as the series $c\sum_n
n^{-\gamma}$ (that is they are both convergent or both divergent). 
This last series diverges for $0<\gamma\leq 1$ and $c>0$. Therefore, in that case, the
second condition in Assumption~\ref{ass:synchro} is
verified. Similarly, the series $\sum_n r_n^2$ has the same behaviour  as 
the series $c\sum_n n^{-2\gamma}$,
which is convergent or not according to $2\gamma>1$ or not.
\end{proof}

In the special case of a single process, the first part of
Theorem~\ref{th3} is in accordance with \cite{pemantle-time-dependent,
  sidorova}. Indeed, \cite[condition (1.4)]{sidorova} corresponds to
$\sum_n (\frac{r_n}{1-r_n})^2<+\infty$, that is $\sum_n r_n^2<+\infty$.  
Moreover, Theorem~\ref{th3} agrees with
\cite[Theorem~2.1]{cri-dai-lou-min}, where $\sum_n r_n^2=+\infty$ is
given as sufficient and necessary condition for the almost sure
asymptotic polarization of the process $(\sum_{l=1}^N Z_{n,l}/N)_n$,
under the mean-field interaction. 
\\

\indent Finally, for the sake of completeness, we conclude this section with the following remark:
\begin{remark}\label{rem-no-atoms-urns}
Given the complete almost sure asymptotic 
synchronization of the system, 
since 
$E[(\vv^\top \vX_{n+1}-\widetilde{Z}_{n})^2|\mathcal{F}_n]\stackrel{a.s.}\to \Phi Z_\infty(1-Z_\infty)$, with $\Phi=\|\vv\|^2>0$, 
we can apply Remark~\ref{rem-TLC-mart} to
$(\widetilde{Z}_n)_n$ so obtaining $P(Z_\infty=z)=0$ for all $z\in (0,1)$, 
provided that there exists a sequence
$(a_n)_n$ such that $a_n\to +\infty$,
$$
a_n^2 \sum_{k\geq n} r_k^2\longrightarrow 1,
\qquad\mbox{and}\qquad
a_n\sup_{k\geq n} r_k\longrightarrow 0\,.
$$ 
For instance, this fact is verified when $\lim_n n^\gamma r_n=c$ with $c>0$ and
$1/2<\gamma\leq 1$ (it is enough to take $a_n=n^{\gamma-1/2}\sqrt{(2\gamma-1))}/c$).
\end{remark}


\section{Estimation of the probability of asymptotic polarization}\label{sec:estimation}

Given the theoretical results about the probability of asymptotic polarization for a network of reinforced stochastic processes, 
our next scope is to provide a procedure in order to estimate this probability, given the observation of the system until a certain time-step.  
In this section, and in the next one, we again face 
the problem in the general setting of a $\mathcal G$-martingale 
$M=(M_n)_{n\geq 0}$ taking values in $[0,1]$ and following  the dynamics 
\eqref{mart-dyn}. Here, the filtration ${\mathcal G}=({\mathcal G}_n)_n$ assumes the meaning of the 
information collected until time-step $n$, i.e. the observed past until time-step $n$, and we aim at providing consistent estimators for 
the conditional probabilities $P_{0,n}=P(M_\infty=0|\mathcal{G}_n)$ and $P_{1,n}=P(M_\infty=1|\mathcal{G}_n)$. \\
\indent We point out that we will not use the lower bound \eqref{eq:prob_fixation_0} since this bound has been obtained evaluating 
the probability that the process converges to the barrier $\{0\}$ by $\delta_n$-\emph{fixation}, 
i.e. evaluating the probability of the event 
$\cap_{n=n_0+m_0}^{+\infty} A_n$. However, in general it could be possible that the process touches the barrier $\{0\}$ in the limit 
without $\delta_n$-fixation. Hence \eqref{eq:prob_fixation_0} would not provide
a consistent estimate for the considered probability.
\\

\indent In the sequel we assume $\sum_n r_n^2<+\infty$, because, as previously shown, 
this condition assures that the probability of asymptotic polarization is 
strictly less than $1$ when $P(0<M_0<1)>0$ (see Proposition \ref{mart-th-base}).  
In this framework, we present some {\em strongly consistent estimators} for the 
conditional probabilities $P_{0,n}=P(M_\infty=0|\mathcal{G}_n)$ and $P_{1,n}=P(M_\infty=1|\mathcal{G}_n)$.
Naturally, the estimation makes sense when the observed value $M_0$ belongs to $(0,1)$ 
(otherwise we trivially have $M_\infty=0$ when $M_0=0$, or $M_\infty=1$ when $M_0=1$, with probability one). 
\\

\indent We start by proving the following result:
\begin{prop}\label{thm:Azuma}
Assume $\sum_n r_n^2< +\infty$. Then, the random variables 
\begin{equation}\label{eq:Azuma}
  U_{0,n}=
\exp\left(-\frac{2M_n^2}{\sum_{k=n}^{\infty} r_k^2}\right)
\quad\mbox{and}\quad
  U_{1,n} = 
\exp\left(-\frac{2(1-M_n)^2}{\sum_{k=n}^{\infty} r_k^2}\right)
\end{equation}
provide almost sure upper bounds for
$P_{0,n}$ and $P_{1,n}$, 
respectively, such that
$$
U_{0,n}-P_{0,n}\stackrel{a.s.}\longrightarrow 0\quad\mbox{and}\quad
U_{1,n}-P_{1,n}\stackrel{a.s.}\longrightarrow 0.
$$
\end{prop}

\begin{proof} We observe that $-r_n(1-M_n)\leq M_n-M_{n+1}=r_n(M_n-Y_{n+1})\leq r_n M_n$ and so,
by Hoeffding's lemma (applied to $M_k-M_{k+1}$ and to $E[\cdot\,|\,\mathcal{G}_k]$ with $a=-r_k(1-M_k)$ and $b=r_kM_k$),
we  have for each $k$ and $t\in\mathbb{R}$
$$
E\left[e^{t(M_k-M_{k+1})}\ |\ \mathcal{G}_k\right]\leq e^{\frac{1}{8} t^2 r_k^2}\quad\mbox{a.s.}\,.
$$
Hence, for each $K>n$ and $t\in\mathbb{R}$, since $M_n-M_K=\sum_{k=n}^{K-1} (M_k-M_{k+1})$, we get 
$$
E[e^{t(M_n-M_K)}\ |\ \mathcal{G}_n]=
E\left[\prod_{k=n}^{K-1} 
E\left[e^{t(M_k-M_{k+1})}\ |\ \mathcal{G}_k\right]\ |\ 
\mathcal{G}_n\right]
\leq e^{\frac{1}{8}t^2\sum_{k=n}^{K-1} r_k^2}\quad\mbox{a.s.}\,.
$$
Taking $K\to +\infty$ (and recalling that $M_K\stackrel{a.s.}\to M_\infty$ 
and $e^{t(M_n-M_K)}\leq e^t$ since $M_n\in [0,1]$ for each $n$), we find
$$
E[e^{t(M_n-M_\infty)}\ |\ \mathcal{G}_n]\leq e^{\frac{1}{8}t^2\sum_{k=n}^\infty r_k^2}\quad\mbox{a.s.}\,.
$$
Therefore, for each $\lambda>0$, we obtain   
\begin{equation*}
\begin{split}
P(M_\infty = 0\ |\ \mathcal{G}_n) &= P(M_n-M_\infty= M_n\ |\ \mathcal{G}_n)\\
&\leq P(M_n-M_\infty\geq M_n\ |\ \mathcal{G}_n)
=P(e^{\lambda(M_n-M_\infty)} \geq e^{\lambda M_n}\ |\ \mathcal{G}_n)\\
& \leq 
e^{-\lambda M_n}E\left[e^{\lambda (M_n-M_\infty)}\ |\ \mathcal{G}_n\right]\\
& \leq 
e^{-\lambda M_n+\frac{1}{8}\lambda^2\sum_{k=n}^\infty r_k^2}\quad\mbox{a.s.}\,.
\end{split}
\end{equation*}
Choosing $\lambda = 4M_n/\sum_{k=n}^\infty r_k^2$ ($>0$ if $M_n>0$, but for $M_n=0$ 
the upper bound is trivial!) in order to minimize the above expression, 
we obtain $P_{0,n}\leq U_{0,n}$ almost surely. Moreover, we know that 
$P_{0,n}=E[I_{0}(M_\infty)|\mathcal{G}_n] \stackrel{a.s.}\to I_{0}(M_\infty)$, 
where $I_0$ denotes the indicator function of the set $\{0\}$, that is $I_0(x)=1$ if $x=0$ or $I_0(x)=0$ otherwise.  
Hence, when $P_{0,n}\to 0$, that is when $M_\infty > 0$,  
by the fact that $\sum_n r_n^2<+\infty$, we have that
$a.s.-\lim_{n\rightarrow \infty}\frac{\sum_{k=n}^\infty r_k^2}{M_n^2}= 
\frac{\lim_{n\rightarrow \infty}\sum_{k=n}^\infty r_k^2}{M_\infty} = 0$ 
and so $U_{0,n}\stackrel{a.s.}\to 0$. When $P_{0,n}\not\to 0$, that is when $P_{0,n}\stackrel{a.s.}\to 1$, 
we also have $U_{0,n}\stackrel{a.s.}\to 1$, because, by construction, 
$P_{0,n}\leq U_{0,n}\leq 1$ a.s. for each $n$. 
\\ \indent 
Analogously, we can compute the upper bound for the barrier $1$, i.e. 
$U_{1,n}$, and prove that 
$U_{1,n}-P_{1,n}\stackrel{a.s.}\to 0$.
\end{proof}

The above upper bounds could be used in order to have consistent 
estimators of $P_{0,n}$ and $P_{1,n}$, respectively. However, at a fixed $n$
the difference between $U_{z, n}$ and $P_{z, n}$, with $z\in\{0,1\}$,
may be large. Therefore, in applications, given the observed past until time-step $n$,
we can get better estimates if we replace 
$U_{z,n}$ by $U'_{z,n,t}=E[U_{z,t}| \mathcal{G}_n]$ with $t>n$.  
Indeed, by Blackwell-Dubins result 
(see \cite[Theorem~2]{black-dubin} or \cite[Lemma~A.2(d)]{Cri}), 
we have 
\begin{equation*}
U'_{z,n,t}-P_{z,n}=E[U_{z,t}-P_{z,t}\,|\,\mathcal{G}_n]\stackrel{a.s.}\longrightarrow 0\quad\mbox{for } t\to +\infty.
\end{equation*}
The quantity $U'_{z,n,t}$ 
can be estimated by simulating a large number $K\in\mathbb{N}$ of realizations of $M_t$ 
based on the observed past $\mathcal{G}_n$ (say $\{M_t^j,\, j=1,\dots,K\}$), then computing the 
corresponding realizations of $U_{z,t}$ by the above formulas (say $\{U_{z,t}^j,\, j=1,\dots,K\}$), 
and finally averaging over these realizations, so that we get 
$$
U'^K_{z,n,t} = \frac{1}{K}\sum_{j=1}^{K}U^j_{z,t},
\qquad\mbox{where}\qquad
\lim_{K\rightarrow \infty} U'^K_{z,n,t} \stackrel{a.s.} = E[U_{z,t}\,|\,\mathcal{G}_n] \stackrel{a.s.}= U'_{z,n,t}.
$$
It is important to notice that the increase of $t$ has only a computational cost
but it does not have anything to do with the increasing of the number of observed data, 
which depends on $n$. Some pratical guidelines on the choice of the time-step $t>n$ are reported
in Section~\ref{app:guideline_time_t} of the Appendix.

\begin{remark} \rm
Note that the random variables~\eqref{eq:Azuma} represent only one of the multiple bounds that can be used to
create analogous consistent estimators of the probability of asymptotic polaritation $P_{0,n}$ and $P_{1,n}$.
Indeed, the methodology presented in this paper works as well with other bounds (see \cite{BouLugMas_book}), 
e.g. Chebychev, Bennet, etc\ldots 
\end{remark}


\section{A confidence interval for $M_\infty$}\label{sec:interval}

In this section we define an asymptotic confidence interval for $M_\infty$ given $\mathcal{G}_n$, i.e.
based on the information collected until time-step $n$. 
Analogously to the estimators $U'^K_{z,n,t}$, for $z\in\{0,1\}$, presented in the previous section,
also this interval is built by 
simulating a large number $K\in\mathbb{N}$ of realizations of $M_t$ (with $t>n$)
based on the observed past $\mathcal{G}_n$.
However, as we will see in Remark \ref{rem:level_without_conditioning}, the confidence interval cannot be simply constructed
using the quantiles of the empirical distribution $M_t$ given $\mathcal{G}_n$
obtained in simulation,
because in that case the desired confidence level $1-\alpha$ would not be guaranteed.
Instead, we define an asymptotic confidence interval for $M_\infty$ as a suitable union of some of these
 three components:  
\begin{itemize}
\item[(i)] an asymptotic confidence interval constructed for the case $0<M_\infty<1$ 
(see Subsection \ref{sec:interval_distrib_M_t_c} for the details); 
\item[(ii)] the barrier set $\{0\}$, in order to include the possible case 
$M_\infty=0$;
\item[(iii)] the barrier set $\{1\}$, in order to include the possible case
$M_\infty=1$.
\end{itemize}
The presence or absence of each one of the above three sets (and the confidence level
chosen for the interval (i)) in the union depends on the estimated probability, given $\mathcal{G}_n$, of 
the events $\{M_\infty=0\}$, $\{M_\infty=1\}$ and $\{0<M_\infty<1\}$, i.e. on  
the estimates of $P_{0,n}$, $P_{1,n}$ and $P_{(0,1),n}=1-P_{0,n}-P_{1,n}$, respectively, proposed 
in the previous section.
\\

\indent Given $\alpha\in (0,1)$, we denote by $I^K_{n,t,1-\alpha}$ 
the asymptotic confidence interval for $M_\infty$ that 
we want to construct in this section. First of all, we  observe that we have the following decomposition:
\[
\begin{aligned}
P(M_\infty \in I^K_{n,t,1-\alpha}\, |\, \mathcal{G}_n) &&=&\ 
P(M_\infty \in I^K_{n,t,1-\alpha}\, |\, \mathcal{G}_n, 0<M_\infty<1)P(0<M_\infty<1\, |\, \mathcal{G}_n)\\
&&+&\ 
P(M_\infty \in I^K_{n,t,1-\alpha}\, |\, \mathcal{G}_n, M_\infty=0)P(M_\infty=0\, |\, \mathcal{G}_n)\\
&&+&\ 
P(M_\infty \in I^K_{n,t,1-\alpha}\, |\, \mathcal{G}_n, M_\infty=1)P(M_\infty=1\, |\, \mathcal{G}_n)\\
&&=&\ 
P(M_\infty \in I^K_{n,t,1-\alpha}\, |\, \mathcal{G}_n, 0<M_\infty<1)P_{(0,1),n}\\
&&+&\ 
P(0 \in I^K_{n,t,1-\alpha}\, |\, \mathcal{G}_n)P_{0,n}\\
&&+&\ 
P(1 \in I^K_{n,t,1-\alpha}\, |\, \mathcal{G}_n)P_{1,n},\\
\end{aligned}
\] 
where we have used that $I^K_{n,t,1-\alpha}$ is based only on $\mathcal{G}_n$ and so 
it does not depend on $M_\infty$. Now, we can consider the consistent estimators, $U'^K_{0,n,t}$ and  
$U'^K_{1,n,t}$, defined in the previous section, and $U'^K_{(0,1),n,t}=1-U'^K_{0,n,t}-U'^K_{1,n,t}$
\footnote{If the time $t>n$ used in the simulations is not high enough to make
these estimators accurate, it is possible that $U'^K_{0,n,t}+U'^K_{1,n,t}>1$;
in that case, we can replace $U'^K_{0,n,t}$ by $\frac{U'^K_{0,n,t}}{U'^K_{0,n,t}+U'^K_{1,n,t}}$, 
$U'^K_{1,n,t}$ by $\frac{U'^K_{1,n,t}}{U'^K_{0,n,t}+U'^K_{1,n,t}}$ and $U'^K_{(0,1),n,t}$ by $0$.}. 
Moreover, let us denote by $I^K_{(0,1),n,t,\theta}$
the asymptotic confidence interval of level $\theta$ for $M_\infty$, based on $\mathcal{G}_n$,  
when $0<M_\infty<1$ (see Subsection \ref{sec:interval_distrib_M_t_c} for details).
Then, we can give the following definition:

\begin{definition}\label{def:confidence:interval}
The asymptotic confidence interval $I^K_{n,t,1-\alpha}$ for $M_\infty$,  
based on $\mathcal{G}_n$,   
is defined as follows:
\begin{itemize}
\item[(1)] $I^K_{n,t,1-\alpha}=\{0\}$, if $U'^K_{0,n,t}\geq 1-\alpha$;
\item[(2)] $I^K_{n,t,1-\alpha}=\{1\}$, if $U'^K_{1,n,t}\geq 1-\alpha$;
\item[(3)] $I^K_{n,t,1-\alpha}= I^K_{(0,1),n,t,\frac{1-\alpha}{U'^K_{(0,1),n,t}}}$, if $U'^K_{(0,1),n,t}\geq 1-\alpha$;
\item[(4)] $I^K_{n,t,1-\alpha}=\{0\}\cup\{1\}$, \\
if $\max\{U'^K_{0,n,t},U'^K_{1,n,t}\}<1-\alpha$,
$U'^K_{0,n,t}+U'^K_{1,n,t}\geq 1-\alpha$ and 
$U'^K_{(0,1),n,t}< \min\{U'^K_{0,n,t},U'^K_{1,n,t}\}$;
\item[(5)] $I^K_{n,t,1-\alpha}=\{0\}\cup I^K_{(0,1),n,t,\frac{1-\alpha-U'^K_{0,n,t}}{U'^K_{(0,1),n,t}}}$, \\
if $\max\{U'^K_{0,n,t},U'^K_{(0,1),n,t}\}<1-\alpha$, 
$U'^K_{0,n,t}+U'^K_{(0,1),n,t}\geq 1-\alpha$ and
$U'^K_{1,n,t}< \min\{U'^K_{0,n,t},U'^K_{(0,1),n,t}\}$;
\item[(6)] $I^K_{n,t,1-\alpha}=\{1\}\cup I^K_{(0,1),n,t,\frac{1-\alpha-U'^K_{1,n}}{U'^K_{(0,1),n,t}}}$, \\
if $\max\{U'^K_{(0,1),n,t},U'^K_{1,n,t}\}<1-\alpha$, $U'^K_{(0,1),n,t}+U'^K_{1,n,t}\geq 1-\alpha$ and
$U'^K_{0,n,t}< \min\{U'^K_{(0,1),n,t},U'^K_{1,n,t}\}$;
\item[(7)] $I^K_{n,t,1-\alpha}=\{0\}\cup\{1\}\cup I^K_{(0,1),n,t,\frac{1-\alpha-U'^K_{0,n,t}-U'^K_{1,n,t}}{U'^K_{(0,1),n,t}}}$,
if all the above conditions do not hold.
\end{itemize}
\end{definition}

We notice that $I^K_{n,t,1-\alpha}$ depends on the $\mathcal{G}_n-$measurable  random variables 
$U'^K_{0,n,t}$, $U'^K_{1,n,t}$ and $U'^K_{(0,1),n,t}$,
which means that the specific form of the interval, i.e. (1)-(7), is selected 
based on the information collected until the time-step $n$.
In addition, we note that the level $\theta$ of the interval $I^K_{(0,1),n,t,\theta}$ is not always the same, 
as it is set so that the global level of the interval $I^K_{n,t,1-\alpha}$ attains (asymtotically in $t$ and $K$) the nominal level $1-\alpha$.
Indeed, if we denote as $C_{n,1-\alpha}$ the asymptotic (in $K$ and $t$) coverage of the interval $I^K_{n,t,1-\alpha}$, i.e. 
$$
C_{n,1-\alpha}\stackrel{a.s.}=\lim_{t,K\rightarrow\infty} P(M_\infty \in I^K_{n,t,1-\alpha} | \mathcal{G}_n),
$$
we can show that, for any value of $P_{0,n}$, $P_{1,n}$ and $P_{(0,1),n}$, we always have $C_{n,1-\alpha}\geq 1-\alpha$.
To this end, we consider the following seven cases, where in each one
the interval $I^K_{n,t,1-\alpha}$ is the
one reported in the corresponding case of the previous Definition \ref{def:confidence:interval}, because it is exactly
the interval that is selected in that case when $t,\, K$ are large (since the choice is based on strongly consistent estimators of 
the probabilities $P_{0,n},\,P_{1,n}$ and $P_{(0,1),n}$): 
\begin{itemize}
\item[(1)] if $P_{0,n}\geq 1-\alpha$, $C_{n,1-\alpha} = 0+1\cdot P_{0,n}+0=P_{0,n} \geq 1-\alpha$;
\item[(2)] if $P_{1,n}\geq 1-\alpha$, $C_{n,1-\alpha} = 0+0+1\cdot P_{1,n}= P_{1,n} \geq 1-\alpha$;
\item[(3)] if $P_{(0,1),n}\geq 1-\alpha$, $C_{n,1-\alpha} = \frac{1-\alpha}{P_{(0,1),n}}\cdot P_{(0,1),n}+0+0 = 1-\alpha$;
\item[(4)] if $\max\{P_{0,n},P_{1,n}\}<1-\alpha$, $P_{0,n}+P_{1,n}\geq 1-\alpha$ and 
$P_{(0,1),n}< \min\{P_{0,n},P_{1,n}\}$,\\
$C_{n,1-\alpha} = 0 + 1\cdot P_{0,n} + 1\cdot P_{1,n} = P_{0,n}+P_{1,n} \geq 1-\alpha$;
\item[(5)] if $\max\{P_{0,n},P_{(0,1),n}\}<1-\alpha$,
$P_{0,n}+P_{(0,1),n}\geq 1-\alpha$ and 
$P_{1,n}< \min\{P_{0,n},P_{(0,1),n}\}$,\\
$C_{n,1-\alpha} = \frac{1-\alpha-P_{0,n}}{P_{(0,1),n}}\cdot P_{(0,1),n} + 1\cdot P_{0,n} + 0= 1-\alpha$;
\item[(6)] if $\max\{P_{(0,1),n},P_{1,n}\}<1-\alpha$,
$P_{(0,1),n}+P_{1,n}\geq 1-\alpha$ and 
$P_{0,n}< \min\{P_{(0,1),n},P_{1,n}\}$, \\
$C_{n,1-\alpha} = \frac{1-\alpha-P_{1,n}}{P_{(0,1),n}}\cdot P_{(0,1),n} + 0 + 1\cdot P_{1,n}= 1-\alpha$;
\item[(7)] if all the above conditions do not hold,\\
$C_{n,1-\alpha} = \frac{1-\alpha-P_{0,n}-P_{1,n}}{P_{(0,1),n}}\cdot P_{(0,1),n} + 1\cdot P_{0,n} + 1\cdot P_{1,n}= 1-\alpha$.
\end{itemize}


\subsection{Confidence interval for the case $0<M_\infty<1$}\label{sec:interval_distrib_M_t_c}

In this section we illustrate the details concerning the asymptotic confidence interval $I^K_{(0,1),n,t,\theta}$ of level $\theta$ 
for $M_\infty$, given $0<M_\infty<1$,
that has been used above for the definition of the interval $I^K_{n,t,1-\alpha}$. 
To avoid unnecessary complications, here we focus on the case when the limit random variable $M_\infty$ 
has no atoms in $(0,1)$ (we recall that a set of sufficient conditions for this scenario are provided in 
Remark~\ref{rem-TLC-mart} of Section~\ref{sec-mart}, which are for instance verified for a network of RSPs as discussed 
 in Remark~\ref{rem-no-atoms-urns} at the end of Section~\ref{sec:barriers}).  
 The interval $I^K_{(0,1),n,t,\theta}$ is based on the quantiles of the conditional distribution of $M_t$ 
given $\mathcal{G}_n$ and $\{0<M_\infty<1\}$,
and we show that 
$$
\lim_{t,K\rightarrow +\infty} P(M_\infty \in I^K_{(0,1),n,t,\theta}\, |\, \mathcal{G}_n, 0<M_\infty<1)
\stackrel{a.s.}= \theta.
$$

First of all, we define the following conditional cumulative distribution functions: 
for any $x\in[0,1]$
\begin{equation*}
\begin{split}
&F_{(0,1),n,\infty}(x) = P(M_\infty \leq x\, |\, \mathcal{G}_n, 0<M_\infty<1)\quad a.s. 
\qquad\mbox{and}\qquad\\
&F_{(0,1),n,t}(x) = P(M_t \leq x\, |\, \mathcal{G}_n, 0<M_\infty<1)\quad a.s.
\end{split}
\end{equation*}
Because of the assumption that $M_\infty$ has no atoms in $(0,1)$, then $F_{(0,1),n,\infty}$ and its inverse are continuous. Moreover, we define 
 the corresponding quantiles of order $\alpha/2$ and $1-\alpha/2$: 
$$
q_{(0,1),n,\infty, \frac{\alpha}{2}} = F_{(0,1),n,\infty}^{-1}\left(\frac{\alpha}{2}\right)
\qquad\mbox{and}\qquad
q_{(0,1),n,\infty, 1-\frac{\alpha}{2}} = F_{(0,1),n,\infty}^{-1}\left(1-\frac{\alpha}{2}\right),
$$
\begin{equation*}
\begin{split}
q_{(0,1),n,t, \frac{\alpha}{2}}&= \min\left\{ x\in[0,1]:\, F_{(0,1),n,t}(x) \geq \frac{\alpha}{2} \right\}
\quad\mbox{and}\\
q_{(0,1),n,t, 1-\frac{\alpha}{2}}&= \min\left\{ x\in[0,1]:\, F_{(0,1),n,t}(x) \geq 1-\frac{\alpha}{2} \right\}.
\end{split}
\end{equation*}
Now, we can set $I_{(0,1),n,t,\theta}$ equal to the interval with extremes $q_{(0,1),n,t, \frac{1-\theta}{2}}$ and 
$q_{(0,1),n,t, \frac{1+\theta}{2}}$, so that we have 
\begin{multline*}
\lim_{t\rightarrow+\infty}P(M_\infty\in I_{(0,1),n,t,\theta}\, |\, \mathcal{G}_n, 0<M_\infty<1) =\\
\begin{split}
 &\lim_{t\rightarrow+\infty}
 P(q_{(0,1),n,t, \frac{1-\theta}{2}}\leq M_\infty \leq q_{(0,1),n,t, \frac{1+\theta}{2}} \, |\, \mathcal{G}_n, 0<M_\infty<1)
\stackrel{a.s.}=\\ 
& 1 - \lim_{t\rightarrow+\infty}\left[
 F_{(0,1),n,\infty}(q_{(0,1),n,t, \frac{1-\theta}{2}})+ 
 1-F_{(0,1),n,\infty}(q_{(0,1),n,t, \frac{1+\theta}{2}})\right]
=\\
&1 - \left[
 F_{(0,1),n,\infty}(q_{(0,1),n,\infty, \frac{1-\theta}{2}})+
 1-F_{(0,1),n,\infty}(q_{(0,1),n,\infty, \frac{1+\theta}{2}})\right]=\\
&1-\frac{1-\theta}{2}-1+\frac{1+\theta}{2} = \theta\,,
\end{split}
\end{multline*}
where we have used the continuity of $F_{(0,1),n,\infty}$ and the fact that 
$q_{(0,1),n,t, \frac{1-\theta}{2}}\to q_{(0,1),n,\infty, \frac{1-\theta}{2}}$ and 
$q_{(0,1),n,t, \frac{1+\theta}{2}}\to q_{(0,1),n,\infty, \frac{1+\theta}{2}}$ as $t\to +\infty$  
(because $M_t\stackrel{a.s.}\to M_\infty$ and $F^{-1}_{(0,1),n,\infty}$ is continuous and 
hence the pseudo-inverse of $F_{(0,1),n,t}$ converges pointwise 
to $F^{-1}_{(0,1),n,\infty}$ for $t\to +\infty$).

\begin{remark}\label{rem:level_without_conditioning}
It is worth highlighting that if we had not conditioned on $\{0<M_\infty<1\}$,
then the cumulative distribution function $F_{n,\infty}(\cdot)$ of $M_\infty$ given $\mathcal{G}_n$
would not be continuous at $0$ and $1$, and so
the probability that $M_\infty$ lies within the quantiles, say $q_{n,t, \frac{1-\theta}{2}}$ and 
$q_{n,t, \frac{1+\theta}{2}}$, defined as $q_{(0,1),n,t, \frac{1-\theta}{2}}$ and 
$q_{(0,1),n,t, \frac{1+\theta}{2}}$ but without the conditioning on $\{0<M_\infty<1\}$, 
would not attain the desired level $\theta$. 
For instance, if $P_{0,n}> \frac{1-\theta}{2}$,  we have $q_{n,\infty, \frac{1-\theta}{2}}=0$, 
but, since $P(0<M_t<1|\mathcal{G}_n)=1$ for any $t$, we always have $q_{n,t, \frac{1-\theta}{2}}>0$, and 
hence\footnote{$F_{n,\infty}(x-)=
F_{n,\infty}(x)-P(M_\infty = x\, |\, \mathcal{G}_n)$.}
$F_{n,\infty}(q_{n,\infty,\frac{1-\theta}{2}}-)=0$, while 
$\lim_t F_{n,\infty}(q_{n,t,\frac{1-\theta}{2}}-)\geq P_{0,n}> 
\frac{1-\theta}{2}$.
Analogously, if $P_{1,n}> \frac{1-\theta}{2}$, we may show that 
$\lim_t 1-F_{n,\infty}(q_{n,t,\frac{1+\theta}{2}})\geq P_{1,n}> \frac{1-\theta}{2}$.
Therefore, we would have 
$\lim_{t\rightarrow\infty}P(q_{n,t, \frac{1-\theta}{2}}\leq M_\infty \leq q_{n,t, \frac{1+\theta}{2}}\, |\, \mathcal{G}_n)
< \theta$.
\end{remark}

In practice we have to estimate the cumulative distribution function 
of $M_t$ conditioned on $\mathcal{G}_n$ and $\{0<M_\infty<1\}$, 
i.e. the function $F_{(0,1),n,t}(\cdot)$ defined above, and
the corresponding quantiles $q_{(0,1),n,t, \frac{1-\theta}{2}}$ and $q_{(0,1),n,t, \frac{1+\theta}{2}}$ needed for 
$I_{(0,1),n,t,\theta}$. In other words, the interval $I^K_{(0,1),n,t,\theta}$ we are constructing corresponds to the interval 
$I_{(0,1),n,t,\theta}$, 
where we replace the two extremes $q_{(0,1),n,t, \frac{1-\theta}{2}}$ and $q_{(0,1),n,t, \frac{1+\theta}{2}}$ 
with their  corresponding estimates  $q^K_{(0,1),n,t, \frac{1-\theta}{2}}$ and $q^K_{(0,1),n,t, \frac{1+\theta}{2}}$. 
\\
\indent More precisely, we observe that, after some easy computations, 
we get for each  $x\in [0,1]$ and $t\geq n$  
$$
F_{(0,1),n,t}(x) = P(M_t \leq x\, |\, \mathcal{G}_n, 0<M_\infty<1)=
\frac{E[\mathbbm{1}_{\{M_t \leq x\}}P_{(0,1),t}\, |\, \mathcal{G}_n]}
{E[P_{(0,1),t}\, |\, \mathcal{G}_n]}\,,
$$
where we recall that $P_{(0,1),t}=P(0<M_\infty<1\, |\, \mathcal{G}_t)=1-P_{0,t}-P_{1,t}$.
Therefore, if we generate a large number $K\in\mathbb{N}$ of realizations of $M_t$  
based on the observed past $\mathcal{G}_n$,  we can approximate    
 $E[\mathbbm{1}_{\{M_t \leq x\}}P_{(0,1),t}\, |\, \mathcal{G}_n]$ 
 by 
$\frac{1}{K}\sum_{j=1}^K \mathbbm{1}_{\{M^j_t \leq x\}}P^j_{(0,1),t}$ 
and  $E[P_{(0,1),t}\, |\, \mathcal{G}_n]=P_{(0,1),n}$  by 
$\frac{1}{K}\sum_{j=1}^K P^j_{(0,1),t}$ (for $K$ large), where $P^j_{(0,1),t}$ can be approximated 
by $U^j_{(0,1),t}$ (for $t$ large), and so 
we can estimate $F_{(0,1),n,t}(x)$ by means of the strongly consistent estimator  
$$
F^K_{(0,1),n,t}(x)=\frac{\sum_{j=1}^K \mathbbm{1}_{\{M^j_t \leq x\}} U^j_{(0,1),t}}
{\sum_{j=1}^K U^j_{(0,1),t}}.
$$
Finally, from the estimator $F^K_{(0,1),n,t}(\cdot)$ we can derive the quantiles 
$q^K_{(0,1),n,t, \frac{1-\theta}{2}}$ and $q^K_{(0,1),n,t, \frac{1+\theta}{2}}$ needed 
for the asymptotic confidence interval $I^K_{(0,1),n,t, \theta}$.

\begin{remark}\label{rem:asymptotic_K_t}
It is important to highlight that the confidence interval $I^K_{(0,1),n,t,\theta}$ 
is to be intended asymptotically in $t$ and $K$, but not in $n$. 
Then, since $n$ represents the data size while $t$ and $K$ represent the simulation size, this means that
improving the confidence of the above interval in order to achieve the desired nominal level $\theta$ 
consists only in a computational cost but not 
in the practical cost of collecting new data.
\end{remark}


\section{Application to a network of reinforced stochastic processes}
\label{sec:RSPs_simulations}

Consider the setting described in Sections~\ref{intro} and~\ref{sec:barriers}. 
We can apply the proposed methodology to the martingale $\widetilde{Z}=(\widetilde{Z}_n)_n$ (see Sec.~\ref{sec:barriers}) 
with the filtration ${\mathcal G}_n={\mathcal F}_n$ defined in Section~\ref{intro}. Indeed, as observed in
Section \ref{sec:barriers}, 
in the case of complete almost sure asymptotic synchronization, 
the random variable $Z_\infty$ can be seen as the almost sure limit of $\widetilde{Z}_n$ and, when
$\sum_n r_n^2 < +\infty$, the 
random variable $U'^K_{z,n,t}$, with $z\in\{0,1\}$, as defined in Section \ref{sec:estimation}
provides a strongly consistent estimator for the probability $P_{z,n}$ that 
$Z_\infty=z$ given the observed past ${\mathcal F}_n$ until time-step $n$, 
i.e.~given the observation of the system until time-step $n$.   
These estimators can be used in applications in order to predict how it is likely 
to have the asymptotic polarization of the network agents 
and then providing a confidence interval for $Z_\infty$, 
given the observation of the system until a certain time-step $n$ 
and provided that the sequence $(r_n)_n$ is known. \\
\indent For instance, let us focus on the case when    
$\lim_n n^{\gamma} r_n = c$ and $\sum_n (r_n - cn^{-\gamma})$ is convergent, 
with $c>0$ and $1/2<\gamma<1$, for which we know that 
$0<P_{z,n}<1$ (see Table \ref{table-summary} in Section \ref{intro}). 
Then, in Figure~\ref{fig:estimation}, for a given choice of $\gamma$ and $c$,  
we have simulated $K=100$ realizations 
of the network until a certain time-step $n$ and, for each of them, we have plotted 
$U'^K_{z,n,t}$ for a chosen $t>n$ 
that satisfies the lower bound provided in the practical guidelines of
Section~\ref{app:guideline_time_t} of the Appendix.

\begin{figure}[htp]
\centering
\includegraphics[scale=0.40]{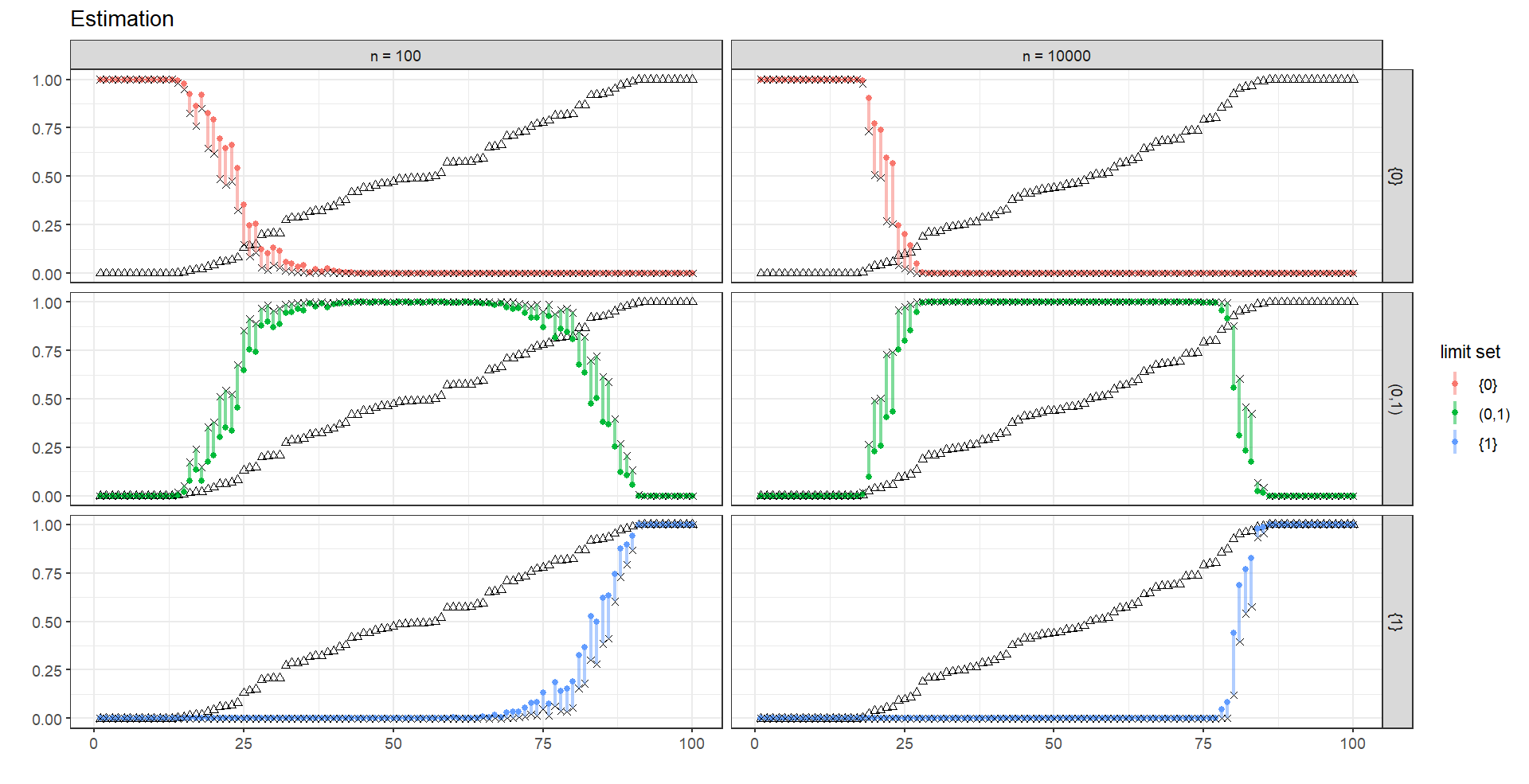}
\caption[]{
Model parameters: $N=3$, $r_n=\frac{c}{(0.1+n)^{\gamma}}$ with $c=1$ and $\gamma=0.75$,
the interaction matrix $W$ is of the mean-field type, precisely 
$w_{l_1,l_2}=\frac{1}{2N} + \frac{1}{2} \delta_{l_1,l_2}$, 
with initial condition $\vZ_0=\frac{1}{2}\vone$. Number of simulations $S=10^2$.\\
\noindent Parameters for the estimating procedure: 
$K=10^2$,
$n=10^2, 10^4$ and $t=n+10^4$ (that satisfies the guidelines  of
Section~\ref{app:guideline_time_t} of the Appendix, e.g. with $\eta=0.2$ and $\epsilon=0.05$).    
In each panel identified by $n$ and a limit set $z=\{0\},(0,1),\{1\}$, 
the triangles correspond to the 
values of $\widetilde{Z}_n$ for the different $S$ simulations. The circular dots
represent the estimate $U'^K_{z,n,t}$ for the different $S$ simulations.
The crosses represent the target values $P_{z,n}\stackrel{a.s.}=\lim_{K,t\rightarrow+\infty}U'^K_{z,n,t}$,
here estimated as $U'^K_{z,n,10^5}$, for the different $S$ simulations.
The vertical segments indicate the differences between the estimates and the corresponding 
target values.
} 
\label{fig:estimation}
\end{figure}

Notice in Figure \ref{fig:estimation} 
how the estimates change according to the values of $\widetilde{Z}_n$ observed in the
different simulations: 
when $\widetilde{Z}_n$ is very close to a barrier, $z=0$ or $z=1$, 
then the corresponding $U'^K_{z,n,t}$ is close to one as well, 
while when $\widetilde{Z}_n$ lies within (0,1) and is far from the barriers, 
then both $U'^K_{0,n,t}$ and $U'^K_{1,n,t}$ are very small, 
so leading $U'^K_{(0,1),n,t}$ to be large instead.
In addition, notice how the estimates of $U'^K_{z,n,t}$ get
more extreme with a larger value of $n$, since
the probability of polarization itself
gets close to 1 or 0 with more observations, i.e.
$U'^K_{z,n,t}\stackrel{a.s.}\longrightarrow_{t,K} P_{z,n}
\stackrel{a.s.}\longrightarrow_n I_z(Z_\infty)\in\{0,1\}$.\\

Figure \ref{fig:confidence_interval} is focused on the
construction of a confidence interval for $Z_\infty$
in the same framework as considered in Figure \ref{fig:estimation}.
In particular, it shows how the asymptotic confidence interval $I^K_{n,t,1-\alpha}$
presented in Section \ref{sec:interval} can be composed by different disjoint sets: $\{0\}$, $\{1\}$
and $I^K_{(0,1),n,t,\theta}$, i.e. the two barriers and the interval for the case $0<Z_\infty<1$.
The specific form assumed by the interval depends on the observation of the system until the time-step $n$, and 
in particular on $\widetilde{Z}_n$.
Indeed, from Figure \ref{fig:confidence_interval} it is evident that when 
$\widetilde{Z}_n$ is very close to a barrier, $z=0$ or $z=1$, 
the interval is only made by that barrier $z$, while when
$\widetilde{Z}_n$ remains inside $(0,1)$ and is far from the barriers
we get a more classical two-sided interval that excludes them.
Naturally, with a larger value of $n$ 
we observe more intervals composed by a single set,
which can be $\{0\}$, $\{1\}$ or the interval $I^K_{(0,1),n,t,\theta}$,
that gets narrower
as a natural consequence of the fact that
$\widetilde{Z}_n\stackrel{a.s.}\longrightarrow Z_\infty$.

\begin{figure}[htp]
\centering
\includegraphics[scale=0.40]{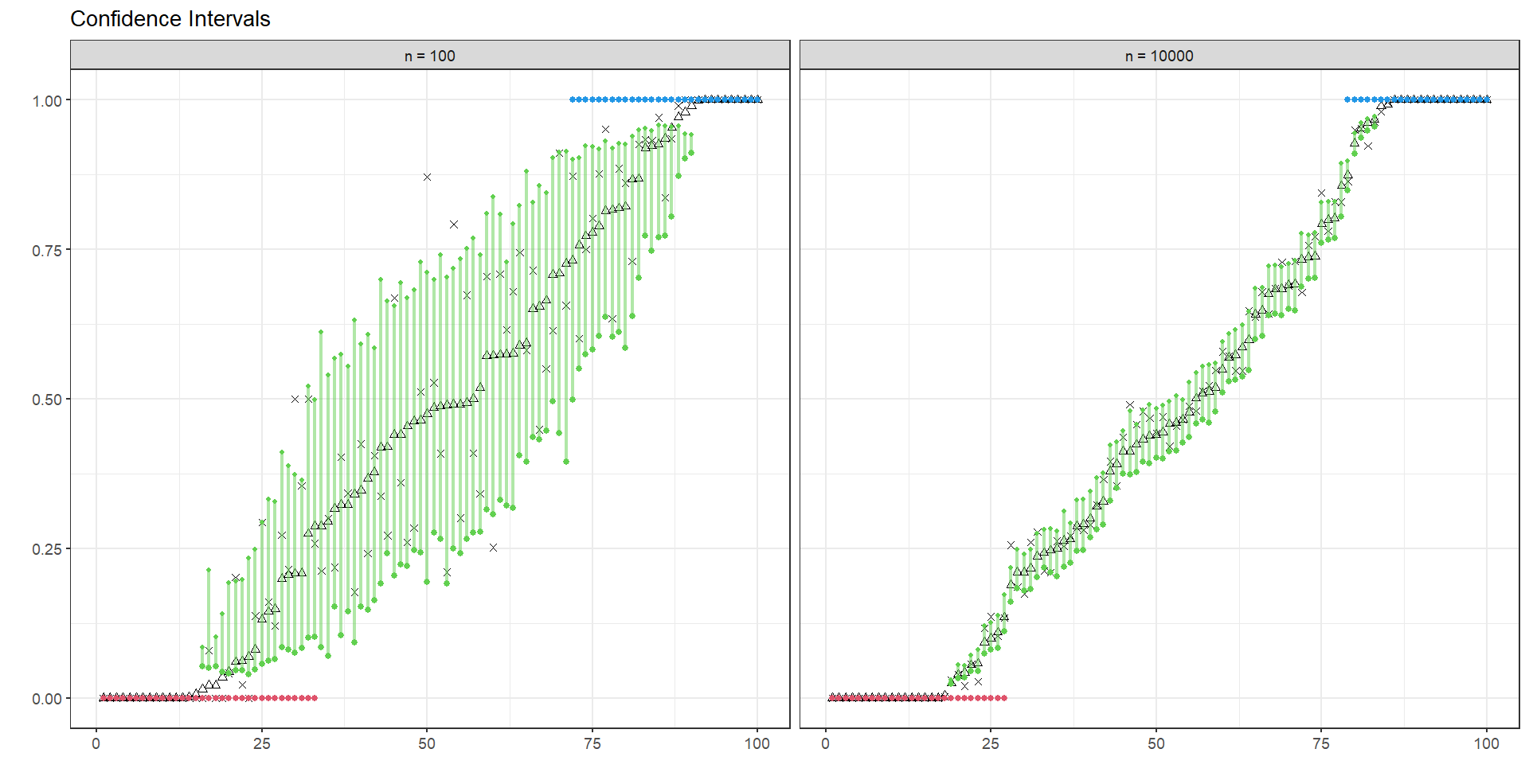}
\caption[]{
Model parameters: $N=3$, $r_n=\frac{c}{(0.1+n)^{\gamma}}$ with $c=1$ and $\gamma=0.75$,
the interaction matrix $W$ is of the mean-field type, precisely 
$w_{l_1,l_2}=\frac{1}{2N} + \frac{1}{2} \delta_{l_1,l_2}$, 
with initial condition $\vZ_0=\frac{1}{2}\vone$. Number of simulations $S=10^2$.\\
\noindent Parameters for the estimating procedure: 
$K=10^2$,
$n=10^2, 10^4$, $t=n+10^4$ (that satisfies the guidelines  of
Section~\ref{app:guideline_time_t} of the Appendix, e.g. with $\eta=0.2$ and $\epsilon=0.05$) and confidence level $1-\alpha=0.95$.
In each panel identified by $n$, the (possible) three parts that can compose the confidence interval
$I^K_{n,t,1-\alpha}$ are reported: the set $\{0\}$ (red), the set $\{1\}$ (blue) and
the interval $I^K_{(0,1),n,t,\theta}$ (green). The triangles correspond to the 
values of $\widetilde{Z}_n$ for the different $S$ simulations. The crosses represent the target values 
$Z_\infty\stackrel{a.s.}=\lim_{t\rightarrow+\infty}\widetilde{Z}_t$,
here estimated as $\widetilde{Z}_{10^5}$, for the different $S$ simulations.
} 
\label{fig:confidence_interval}
\end{figure}

\begin{remark}\rm
We observe that the estimating procedure for the probability of asymptotic polarization 
and the construction of the asymptotic confidence interval for $Z_\infty$ can be adapted 
to the case when only the asymptotic behaviour of $(r_n)$ is known and 
the observable variables are the agents' actions $X_{n,\ell}$. In that case, 
we have to replace the random variables $Z_{n,\ell}$ by the empirical means 
$\sum_{k=1}^n X_{k,\ell}/n$ or by suitable weighted empirical means 
\cite{ale-cri-ghi-MEAN, ale-cri-ghi-WEIGHT-MEAN}. 
\\[15pt]
\end{remark}


\noindent{\bf Declaration}\\
\noindent All the authors contributed equally to the present work.\\

\noindent{\bf Funding Sources}\\
\noindent Irene Crimaldi is partially supported by the Italian ``Programma di Attività
Integrata” (PAI), project “TOol for Fighting FakEs” (TOFFE) funded by IMT
School for Advanced Studies Lucca. Giacomo Aletti and Irene Crimaldi are
partially supported by the project ``Optimal and adaptive designs for modern
medical experimentation” funded by the Italian Government (MIUR, PRIN 2022).
\\

\noindent{\bf Acknowledgements}\\
\noindent The authors 
thank the referees for carefully reading the manuscript and 
for their suggestions to improve it.



\appendix

\renewcommand{\thesection}{\Alph{section}}

\section{Some auxiliary results}\label{app:proofs}

We start with the proof of Lemma~\ref{lem-equivalent}. 

\begin{proof}[Proof of Lemma~\ref{lem-equivalent}]
a) In order to prove \eqref{eq-div-serie}, 
we note that 
$$
\sum_{n\geq 0} \frac{r_n}{\sum_{k=0}^n r_k}
= 1 +
\sum_{n \geq 1} \frac{r_{n}}{\sum_{k=0}^{n} r_k}
$$
and, since $\sum_{k=0}^{n-1} r_k/\sum_{k=0}^n r_k=(1+r_n/\sum_{k=0}^{n-1}r_k)^{-1}\to 1$ when $\sum_n r_n=+\infty$, 
we have that the series   
$\sum_{n \geq 1} \frac{r_{n}}{\sum_{k=0}^{n} r_k}$ has the same behaviour as the series 
$\sum_{n \geq 1} \frac{r_{n}}{\sum_{k=0}^{n-1} r_k}$ (i.e.~they are both convergent or both divergent). 
In order to prove that the last series diverges to $+\infty$, we observe that  
\begin{align*}
\sum_{n \geq 1} \frac{r_{n}}{\sum_{k=0}^{n-1} r_k}
=\sum_{n \geq 1} \int_{\sum_{k=0}^{n-1}r_k}^{\sum_{k=0}^{n}r_k}\frac{1}{\sum_{k=0}^{n-1} r_k}dx
\geq \sum_{n \geq 1} \int_{\sum_{k=0}^{n-1}r_k}^{\sum_{k=0}^{n}r_k}\frac{1}{x}dx=\int_{r_0}^{+\infty} \frac{1}{x}\, dx =+\infty.
\end{align*}

\noindent
b) The relation \eqref{eq-equivalence-sumr-logs} follows by combining the following inequalities:
\begin{equation}\label{eq-b-prima}
\begin{split}
\sum_{k=0}^n r_k&=\sum_{k=0}^n\frac{\alpha_{k+1}}{s_{k+1}}=\sum_{k=0}^n\int_{s_k}^{s_{k+1}} \frac{1}{s_{k+1}}\,dx
\\
&\leq \sum_{k=0}^n\int_{s_k}^{s_{k+1}} \frac{1}{x}\,dx=\int_{s_0}^{s_{n+1}} \frac{1}{x}\,dx =\ln (s_{n+1})-\ln(s_0)
\leq \ln(s_{n+1})
\end{split}
\end{equation}
and
\begin{equation}\label{eq-div-serie-3}
\begin{split}
\frac{1}{1-\sup_n r_n}\sum_{k=0}^n r_k& 
\geq \sum_{k=0}^n \frac{r_k}{1-r_k}=\sum_{k=0}^n\frac{\alpha_{k+1}}{s_{k}}=\sum_{k=0}^n\int_{s_k}^{s_{k+1}} \frac{1}{s_{k}}\,dx
\\
&\geq \sum_{k=0}^n\int_{s_k}^{s_{k+1}} \frac{1}{x}\,dx=
\int_{s_0}^{s_{n+1}} \frac{1}{x}\,dx =
\ln (s_{n+1})-\ln(s_0).
\end{split}
\end{equation}
Indeed, we get 
\begin{equation}\label{eq-new}
(1-\sup_n r_n)\left(1 - \frac{\ln(s_0)}{\ln(s_{n+1})}\right)\leq 
\frac{\sum_{k=0}^n r_k}{\ln(s_{n+1})} \leq 1 
\end{equation}
and, since the limit of $\ln(s_0)/\ln(s_{n+1})$ is strictly smaller than $1$ if $\sup_n r_n>0$ (i.e.~if 
the numbers $\alpha_n$ are not all equal to zero),   
there exists a real constant $c>0$ such that 
$$
c\leq \liminf_n \frac{\sum_{k=0}^n r_k}{\ln(s_{n+1})} \leq \limsup_n \frac{\sum_{k=0}^n r_k}{\ln(s_{n+1})} \leq 1\,.
$$
\noindent 
c) It is immediate to see that \eqref{eq-cond-zero} implies \eqref{eq-cond-zero2} since
$(1-x)\leq e^{-x}$ $\forall x>0$. The converse is also trivial when $\sum_n r_n < +\infty$.
Indeed, \eqref{eq-cond-zero2} is always verified when
$\sum_n r_n < +\infty$ and
$\sum_n r_n < +\infty$ trivially implies 
$\sum_n r_n^2 < +\infty$ since $r_n< 1$ $\forall n$.
\\
\indent Now, we focus on the case when $\sum_n r_n = +\infty$. In that case, there exists $N>0$ such that
$\sum_{k=0}^{N-1} r_k \geq 1$. Then, since $x \mapsto xe^{-x}$ is decreasing on $[1,\infty)$, under \eqref{eq-cond-zero2}
we obtain
\[
\begin{aligned}
\sum_{n\geq N} r_n^2 &\leq \sum_{n\geq N} r_n \Big( C e^{-\sum_{k=0}^n r_k} \sum_{k=0}^n r_k  \Big) = 
C \sum_{n\geq N} \int_{\sum_{k=0}^{n-1} r_k}^{\sum_{k=0}^n r_k} e^{-\sum_{k=0}^n r_k} \sum_{k=0}^n r_k dx\\
&\leq
C \sum_{n\geq N}\int_{\sum_{k=0}^{n-1} r_k}^{\sum_{k=0}^n r_k}e^{-x}xdx 
\leq C \int_1^{+\infty}e^{-x}xdx  < +\infty,
\end{aligned}
\]
and so the first part of point c) is proven, 
that is \eqref{eq-cond-zero2} implies $\sum_n r_n^2<+\infty$ (and so $r_n\to 0$). 
As a consequence of this first result, we have 
$$
\sum_{k=0}^n \frac{r_k}{1-r_k}-\sum_{k=0}^n r_k = \sum_{k=0}^n \frac{r_k^2}{1-r_k}<+\infty\,. 
$$
Therefore, using the inequality \eqref{eq-div-serie-3} for $\sum_{k=0}^n \frac{r_k}{1-r_k}$, we obtain 
$$
\sum_{k=0}^n r_k \geq \ln (s_{n+1})-\ln(s_0)-\sum_{k=0}^n \frac{r_k^2}{1-r_k}\,,
$$
that is $s_{n+1}=O\left(\exp(\sum_{k=0}^n r_k)\right)$. Then, \eqref{eq-cond-zero2} also implies \eqref{eq-cond-zero}: 
indeed, for suitable real constants $C_1>0$ and $C_2>0$, we have 
\begin{align*}
\frac{r_n}{\prod_{k=0}^{n}(1-r_k)} & = r_n \frac{s_{n+1}}{s_0} \\
& \leq \Big( C_1 e^{-\sum_{k=0}^n r_k}  \sum\nolimits_{k=0}^n r_k \Big) 
( C_2 e^{\sum_{k=0}^n r_k } ) \leq C_1C_2 \sum\nolimits_{k=0}^n r_k .
\end{align*}
Recalling that 
$\alpha_{n+1}=s_0\frac{r_n}{\prod_{k=0}^n (1-r_k)}$ and  using \eqref{eq-new}, we get that 
\eqref{eq-cond-zero} is equivalent to \eqref{eq-cond-zero3}, i.e.~$\alpha_{n+1}=O(\ln(s_{n+1}))$. 
Finally, recalling that 
$s_{n+1}=s_0/\prod_{k=0}^n (1-r_k)$ and using \eqref{eq-new}, 
we also obtain that \eqref{eq-cond-zero} is equivalent to \eqref{eq-cond-zero4}, 
i.e.~$r_n=O(\ln(s_{n+1})/s_{n+1})$. 
\end{proof}

For reader's convenience, we here recall some general results. 
In particular, the first one is a little generalization of \cite[Lemma A.4]{ale-cri-ghi}. 

\begin{lemma}\label{lemma-product} 
Given $1/2<\gamma\leq 1$ and $c>0$, let $(r_n)_n$ be a sequence of
real numbers such that $0\leq r_n<1$, 
\begin{equation*}
\lim_n n^\gamma r_n=c\quad\mbox{ and }\quad 
\sum_n (r_n-cn^{-\gamma})\;\mbox{ is convergent}\,. 
\end{equation*}
Then we have 
\begin{equation*}
\prod_{k=0}^n (1-r_k)\ =\ \begin{cases}
O\left(\exp \left[ -\frac{c }{1-\gamma}n^{1-\gamma} \right]\right)
& \mbox{for } 1/2<\gamma<1, \\
O\left(n^{- c}\right) & \mbox{for } \gamma=1\,.
\end{cases}
\end{equation*}
\end{lemma}
We omit the proof of the above lemma because it is exactly the same as the one of 
\cite[Lemma A.4]{ale-cri-ghi}.

\begin{theorem}[{\cite[Theorem~46, p.~40]{del-mey}}]\label{del-mey-result-app}
  Let $(Y_n)_n$ be a sequence of
  positive (\textit{i.e.}~non-negative) random variables, adapted to a
  filtration $\mathcal G$. Then the set $\{\sum_n
  E[Y_{n+1}\vert \mathcal{G}_{n}]<+\infty\}$ is almost surely contained in
  the set $\{\sum_n Y_n<+\infty\}$. If the random variables $Y_n$ are
  uniformly bounded by a constant, then these two sets are almost
  surely equal.
\end{theorem}


\section{Practical guidelines on the choice of the time $t$}\label{app:guideline_time_t}

In this section we provide a practical guideline
on the choice of the time-step $t>n$ that has to be used in the above described estimating procedure.
\\
\indent For instance, consider the estimator $U'^K_{0,n,t}=\frac{1}{K}\sum_{j=1}^K U^j_{0,t}$ of $P_{0,t}$ 
(analogous arguments can be made for the estimator $U'^K_{1,n,t}$ of $P_{1,t}$) 
and focus on a single realization $U^j_{0,t}$
and, in particular, on its expression in \eqref{eq:Azuma}, where it is
defined as a function of the ratio $\sum_{k=t}^\infty r_k^2/(M_t^j)^2$. 
Since by Proposition \ref{thm:Azuma} $U^j_{0,t} \stackrel{a.s.}\longrightarrow_t I_{0}(M^j_\infty)$, 
this ratio must tend to infinity a.s. on the set $\{M^j_\infty =0\}$. On the other hand, 
since $\sum_n r_n^2<+\infty$, we obviously have that it tends to zero a.s. on the set $\{0<M^j_\infty\leq 1\}$. 
Moreover, we notice that $(M_t^j)^2$ is bounded by $1$, while
$\sum_{k=t}^\infty r_k^2$, although it always tends to zero, can be very large for small $t$ (much larger than $1$).
As a consequence, if $U^j_{0,t}$ is computed when the time-step $t$ is small, 
the ratio $\sum_{k=t}^\infty r_k^2/(M_n^j)^2$ (and so $U^j_{0,t}$) may be always large,
regardless of the value of $M^j_t$, i.e. regardless of the information 
contained in $\mathcal{G}_n$ used to generate $M^j_t$.
To address this  issue, we can impose some initial conditions in order to ensure that:
\begin{itemize}
\item[(a)] $U^j_{0,t}$ is small when $M^j_{t}$ is not very close to $0$;
\item[(b)] $U^j_{1,t}$ is small when $M^j_{t}$ is not very close to $1$.
\end{itemize} 
Formally, we can fix $\epsilon>0$ and $\eta>0$ such that
$U^j_{0,t}<\epsilon$ when $M^j_{t}>\eta$ and, analogously, 
$U^j_{1,t}<\epsilon$ when $M^j_{t}<1-\eta$.
These constraints provide us with a condition on the minimum time-step $t_{\min}$ that we should take 
to compute the previous estimators: indeed, we have that $t_{\min}$ should be the minimum integer $t$ such that
$$\exp\left(-2\frac{\eta^2}{\sum_{k=t}^\infty r_k^2}\right)<\epsilon,\ \mbox{that is}\ 
\sum_{k=t}^\infty r_k^2 < \frac{2\eta^2}{\ln\left(\frac{1}{\epsilon}\right)}.
$$

\section{Periodicity of a matrix}\label{app:periodic_matrix}

Let $A$ be a non-negative $N\times N$ square matrix such that $A^\top \vone = \vone$.
Then, for any element $l\in V=\{1,\dots, N\}$, we can define the \emph{period} of $l$, say $d(l)$, 
as the  greatest common divisor $n\in\mathbb{N}$ such that $A^n_{l,l} > 0$.
Then, if $A$ is also irreducible, all the elements will have the same period and so we can define the period of the matrix $A$ as 
$d=d(1)=\dots=d(N)$. A matrix with period $d=1$ is called \emph{aperiodic}, otherwise a matrix with period $d\geq2$ is called \emph{periodic}.


\begin{thebibliography}{10}
\bibitem{freed73}
D.~Freedman.
\newblock Another note on the {B}orel-{C}antelli lemma and the strong law, with
  the {P}oisson approximation as a by-product.
\newblock {\em Ann. Probability}, 1:910--925, 1973.
\end{thebibliography}

\end{document}